\documentclass[a4paper,11pt,reqno]{amsart}
\usepackage{amsmath,amsfonts,amsthm,amssymb,color}
\usepackage[T1]{fontenc}
\usepackage{pdfsync}
\usepackage{csquotes}
\usepackage{graphicx}
\usepackage{pstricks}
\usepackage{lmodern}

\usepackage{tikz}

\newcommand{\mathbbm}[1]{\text{\usefont{U}{bbm}{m}{n}#1}}
\newcommand{\be}{\beta}

\newcommand{\1}{{\bf 1}}

\newcommand{\xiti}{\tilde{\xi}}
\newcommand{\etati}{\tilde{\eta}}
\newcommand{\lati}{\tilde{\la}}

\newcommand{\yti}{\tilde{y}}

\newcommand{\Hti}{\tilde{H}}

\newcommand{\R}{\mathbb R}

\newcommand{\ca}{\mathcal A}

\newcommand{\cac}{\mathcal C}
\newcommand{\cd}{\mathcal D}

\newcommand{\cf}{\mathcal F}

\newcommand{\cj}{\mathcal J}

\newcommand{\cp}{\mathcal P}

\newcommand{\cw}{\mathcal W}

\newcommand{\al}{\alpha}

\newcommand{\ga}{\gamma}
\newcommand{\gga}{\Gamma}
\newcommand{\ka}{\kappa}
\newcommand{\la}{\lambda}

\newcommand{\si}{\sigma}

\newcommand{\vp}{\varphi}

\newtheorem{theorem}{Theorem}[section]

\newtheorem{corollary}[theorem]{Corollary}

\newtheorem{definition}[theorem]{Definition}

\newtheorem{lemma}[theorem]{Lemma}

\newtheorem{proposition}[theorem]{Proposition}

\theoremstyle{remark}
\newtheorem{remark}[theorem]{Remark}

\pgfdeclareshape{crosscircle}
{
  \inheritsavedanchors[from=circle] 
  \inheritanchorborder[from=circle]
  \inheritanchor[from=circle]{north}
  \inheritanchor[from=circle]{north west}
  \inheritanchor[from=circle]{north east}
  \inheritanchor[from=circle]{center}
  \inheritanchor[from=circle]{west}
  \inheritanchor[from=circle]{east}
  \inheritanchor[from=circle]{mid}
  \inheritanchor[from=circle]{mid west}
  \inheritanchor[from=circle]{mid east}
  \inheritanchor[from=circle]{base}
  \inheritanchor[from=circle]{base west}
  \inheritanchor[from=circle]{base east}
  \inheritanchor[from=circle]{south}
  \inheritanchor[from=circle]{south west}
  \inheritanchor[from=circle]{south east}
  \inheritbackgroundpath[from=circle]
  \foregroundpath{
    \centerpoint%
    \pgf@xc=\pgf@x%
    \pgf@yc=\pgf@y%
    \pgfutil@tempdima=\radius%
    \pgfmathsetlength{\pgf@xb}{\pgfkeysvalueof{/pgf/outer xsep}}%
    \pgfmathsetlength{\pgf@yb}{\pgfkeysvalueof{/pgf/outer ysep}}%
    \ifdim\pgf@xb<\pgf@yb%
      \advance\pgfutil@tempdima by-\pgf@yb%
    \else%
      \advance\pgfutil@tempdima by-\pgf@xb%
    \fi%
    \pgfpathmoveto{\pgfpointadd{\pgfqpoint{\pgf@xc}{\pgf@yc}}{\pgfqpoint{-0.707107\pgfutil@tempdima}{-0.707107\pgfutil@tempdima}}}
    \pgfpathlineto{\pgfpointadd{\pgfqpoint{\pgf@xc}{\pgf@yc}}{\pgfqpoint{0.707107\pgfutil@tempdima}{0.707107\pgfutil@tempdima}}}
    \pgfpathmoveto{\pgfpointadd{\pgfqpoint{\pgf@xc}{\pgf@yc}}{\pgfqpoint{-0.707107\pgfutil@tempdima}{0.707107\pgfutil@tempdima}}}
    \pgfpathlineto{\pgfpointadd{\pgfqpoint{\pgf@xc}{\pgf@yc}}{\pgfqpoint{0.707107\pgfutil@tempdima}{-0.707107\pgfutil@tempdima}}}
  }
}
\makeatother

\colorlet{symbols}{blue!90!black}
\colorlet{testcolor}{green!60!black}

\def\symbol#1{\textcolor{symbols}{#1}}
\def\1{\mathbf{\symbol{1}}}

\usetikzlibrary{shapes.misc}
\usetikzlibrary{shapes.symbols}
\usetikzlibrary{snakes}
\usetikzlibrary{decorations}
\usetikzlibrary{decorations.markings}

\def\drawx{\draw[-,solid] (-3pt,-3pt) -- (3pt,3pt);\draw[-,solid] (-3pt,3pt) -- (3pt,-3pt);}
\tikzset{
	root/.style={circle,fill=testcolor,inner sep=0pt, minimum size=2mm},
	dot/.style={circle,fill=black,inner sep=0pt, minimum size=1mm},
	var/.style={circle,fill=black!10,draw=black,inner sep=0pt, minimum size=2mm},
	dotred/.style={circle,fill=black!50,inner sep=0pt, minimum size=2mm},
	generic/.style={semithick,shorten >=1pt,shorten <=1pt},
	dist/.style={ultra thick,draw=testcolor,shorten >=1pt,shorten <=1pt},
	testfcn/.style={ultra thick,testcolor,shorten >=1pt,shorten <=1pt,<-},
	testfcnx/.style={ultra thick,testcolor,shorten >=1pt,shorten <=1pt,<-,
		postaction={decorate,decoration={markings,mark=at position 0.6 with {\drawx}}}},
	kprime/.style={semithick,shorten >=1pt,shorten <=1pt,densely dashed,->},
	kprimex/.style={semithick,shorten >=1pt,shorten <=1pt,densely dashed,->,
		postaction={decorate,decoration={markings,mark=at position 0.4 with {\drawx}}}},
	kernel/.style={semithick,shorten >=1pt,shorten <=1pt,->},
	multx/.style={shorten >=1pt,shorten <=1pt,
		postaction={decorate,decoration={markings,mark=at position 0.5 with {\drawx}}}},
	kernelx/.style={semithick,shorten >=1pt,shorten <=1pt,->,
		postaction={decorate,decoration={markings,mark=at position 0.4 with {\drawx}}}},
	kernel1/.style={->,semithick,shorten >=1pt,shorten <=1pt,postaction={decorate,decoration={markings,mark=at position 0.45 with {\draw[-] (0,-0.1) -- (0,0.1);}}}},
	kernel2/.style={->,semithick,shorten >=1pt,shorten <=1pt,postaction={decorate,decoration={markings,mark=at position 0.45 with {\draw[-] (0.05,-0.1) -- (0.05,0.1);\draw[-] (-0.05,-0.1) -- (-0.05,0.1);}}}},
	kernelBig/.style={semithick,shorten >=1pt,shorten <=1pt,decorate, decoration={zigzag,amplitude=1.5pt,segment length = 3pt,pre length=2pt,post length=2pt}},
	rho/.style={dotted,semithick,shorten >=1pt,shorten <=1pt},
	renorm/.style={shape=circle,fill=white,inner sep=1pt},
	labl/.style={shape=rectangle,fill=white,inner sep=1pt},
	xi/.style={circle,fill=symbols!10,draw=symbols,inner sep=0pt,minimum size=1.2mm},
	xix/.style={crosscircle,fill=symbols!10,draw=symbols,inner sep=0pt,minimum size=1.2mm},
	xib/.style={circle,fill=symbols!10,draw=symbols,inner sep=0pt,minimum size=1.6mm},
	xibx/.style={crosscircle,fill=symbols!10,draw=symbols,inner sep=0pt,minimum size=1.6mm},
	not/.style={circle,fill=symbols,draw=symbols,inner sep=0pt,minimum size=0.5mm},
	>=stealth,
	}
\makeatletter
\def\DeclareSymbol#1#2#3{\expandafter\gdef\csname MH@symb@#1\endcsname{\tikz[baseline=#2,scale=0.15,draw=symbols]{#3}}\expandafter\gdef\csname MH@symb@#1s\endcsname{\scalebox{0.7}{\tikz[baseline=#2,scale=0.15,draw=symbols]{#3}}}}
\def\<#1>{\csname MH@symb@#1\endcsname}
\makeatother

\DeclareSymbol{circle}{0.5}{\draw (0,0.7) node[xi] {};}
\DeclareSymbol{line}{0.5}{\draw (0,0.2) node[not] {} -- (0,1.4) node[not] {};}

\DeclareSymbol{Psi}{0.5}{\draw (0,0) node[not] {} -- (0,1.5) node[xi] {};}
\DeclareSymbol{Psi2}{0.5}{\draw (-0.8,1) node[xi] {} -- (0,0) node[not] {} -- (0.8,1) node[xi] {};}
\DeclareSymbol{IPsi2}{0}{\draw (0,1) -- (0.8,2.2) node[xi] {};\draw (0,-0.25) node[not] {} -- (0,1) node[not] {} -- (-0.8,2.2) node[xi] {};}
\DeclareSymbol{PsiIPsi2}{0}{\draw (0,1) -- (0.8,2.2) node[xi] {};\draw (0,-0.25) node[not] {} -- (0,1) node[not] {} -- (-0.8,2.2) node[xi] {};\draw (0,-0.25) node[not] {} -- (0.8,1) node[xi] {};}

\newcommand{\luxor}{\mathbf{\Psi}}

\newcommand{\cherry}{\mathbf{\Psi}^{\mathbf{2}}}

\date{\today}

\title{Study of a fractional stochastic heat equation}

\numberwithin{equation}{section}
\begin{document}

\begin{center}
{\large\textbf{
Study of a fractional stochastic heat equation}}\\~\\ Nicolas Schaeffer\footnote{Institut \'Elie Cartan, Universit\' e de Lorraine, BP 70239, 54506 Vandoeuvre-l\`es-Nancy, France. }
\end{center}

\smallskip
{\small \noindent {\bf Abstract.}}
In this article, we study a $d$-dimensional stochastic  nonlinear heat equation (SNLH) with a quadratic nonlinearity, forced by a fractional space-time white noise:
\begin{equation*}
\left\{
\begin{array}{l}
\partial_t u-\Delta u= \rho^2 u^2 + \dot B \, , \quad  t\in [0,T] \, , \, x\in \R^d \, ,\\
u_0=\phi\, .
\end{array}
\right.
\end{equation*}
Two types of regimes are exhibited, depending on the ranges of the Hurst index $H=(H_0,...,H_d)$ $\in (0,1)^{d+1}$. In particular, we show that the local well-posedness of (SNLH) resulting from the Da Prato-Debussche trick, is easily obtained when $2H_0+\sum_{i=1}^{d}H_i >d$. On the contrary, (SNLH) is much more difficult to handle when $2H_0+\sum_{i=1}^{d}H_i \leq d$. In this case, the model has to be interpreted in the Wick sense, thanks to a time-dependent renormalization. Helped with the regularising effect of the heat semigroup, we establish local well-posedness results for (SNLH) for all dimension $d\geq1.$

\tableofcontents
\section{Introduction and main results}

\
\subsection{General introduction}
This paper is devoted to the study of a heat equation with a quadratic nonlinearity, driven by an additive fractional space-time white noise forcing. More precisely, we consider the following nonlinear stochastic heat model:
\begin{equation}\label{dim-d}
\left\{
\begin{array}{l}
 \partial_t u-\Delta u= \rho^2 u^2 + \dot B \, , \quad  t\in [0,T] \, , \, x\in \R^d \, ,\\
u(0,.)=\phi\, ,
\end{array}
\right.
\end{equation}
where $\phi$ is a (deterministic) initial condition living in some appropriate Sobolev space and $\rho$ : $\mathbb{R}^d \rightarrow~ \mathbb{R}$ is a smooth compactly-supported function fixed once and for all. Before going any further, let us specify the roughness of the stochastic term, namely, $\dot{B}:=\partial_t \partial_{x_1}...\partial_{x_d}B$ where $B=B^H$ is a space-time fractional Brownian motion of Hurst index $H=(H_0,H_1,...,H_d) ~\in (0,1)^{d+1}.$ The definition of this process is the following:

\smallskip

\begin{definition}\label{white}
Fix $d \geq 1$ a space dimension, $T\geq 0$ a positive time and $(\Omega,\mathcal{F},\mathbb{P})$ a complete filtered probability space. On this space, and for any fixed $H=(H_0,H_1,...,H_d) \in (0,1)^{d+1}$,  a centered Gaussian process $B:\Omega \times ([0,T]\times \R^d) \to \R$ is said to be a space-time fractional Brownian motion of Hurst index $H$ if its covariance function is given by the formula: 
$$\mathbb{E} \big[ B(s,x_1,...,x_d)B(t,y_1,...,y_d)\big] = R_{H_0}(s,t)\prod_{i=1}^{d} R_{H_i}(x_i,y_i) \, , \quad  s,t\in [0,T] \, , \ x,y\in \R^d \, ,$$
where
$$R_{H_i}(x,y):=\frac12 (|x|^{2H_i}+|y|^{2H_i}-|x-y|^{2H_i}) \ .$$
\end{definition}

Stochastic partial differential equations (SPDEs) perturbed by fractional noise terms have become increasingly popular in the recent years. Indeed, these equations can be interpreted as a more flexible alternative to classical SPDEs driven by a space-time white noise, and consequently they can be used to model  more complex physical phenomena which are subjects to random perturbations. For instance, fractional noises appear in lots of practical situations, whether in biophysics (see \cite{Kou}), in the study of financial time series (see \cite{bayraktar}) or simply in electrical
engineering (see \cite{denk}).

The major motivation of this work is that equation (\ref{dim-d}) is both a physical and mathematically challenging model. In fact, it can be seen as a stochastic reaction diffusion model with $p=2$. Recall that the equation:
\begin{equation}\label{diffusion}
 \partial_t u=\Delta u + u^p, \quad p>1
\end{equation}
 is known under the name of Fujita model and has been the subject of many questions. It has already been established (see \cite{brezis,weiss}) that when $\phi \in L^q(\mathbb{R}^d)$ with $q\geq1$ and $q > \frac{d(p-1)}{2}$, there exists a constant $T = T(\phi) > 0$ and a unique
function $u \in \mathcal{C}([0, T], L^q(\mathbb{R}^d))$ that is a classical solution to (1.1) on $[0, T]\times \mathbb{R}^d$ whereas when $q < \frac{d(p-1)}{2}$, there is no general theory of existence. In this work, we introduce a cut-off function $\rho$ to bring back computations to compact domains, a fractional random perturbation $\dot B$ and we are interested in the well-posedness of the associated equation (\ref{dim-d}).

Two classical obstacles have to be taken over. First of all, as often in the theory of SPDEs, a standard problem is the roughness of the fractional noise which prevents us from solving directly the associated linear part of the equation, namely:
\begin{equation}
 \left\{
    \begin{array}{ll}
  \partial_t\<Psi>-\Delta \<Psi> = \dot B, \hskip 0.3 cm t \in [0,T], \hspace{0,2cm} x \in \mathbb{R}^d,   \\
 \  \<Psi>(0,.)=0.
    \end{array}
\right.
\end{equation}
Then, the quadratic nonlinearity will require us to give a precise meaning to the notion of solution. Before going into details, let us recall that, over the last decade, we have seen numerous developments in the study of singular stochastic PDEs, in particular in the parabolic setting. Remember that there are three main components to consider: the space dimension $d$, the nature of the nonlinearity and the irregularity of the noise. For example, polynomial nonlinearities \cite[Section 3]{E} have been progressively replaced by sinusoidal and exponential ones. As an application of the theory of regularity structures, Hairer-Shen \cite{hairer-shen} and Chandra-Hairer-Shen \cite{chandra} studied the following parabolic sine-Gordon model on $\mathbb{T}^2$:

\begin{equation}\label{sine-Gordon}
\partial_t u = \frac{1}{2}\Delta u +\sin(\be u)+  \xi,
\end{equation}
where $\xi$ is a space-time white noise and proved that the local well-posedness depends essentially on the value of the key parameter $\be ^2 >0$.
Oh, Robert and Wand \cite{oh} focused on the parabolic equation in dimension $2$: 
\begin{equation}
\partial_t u +\frac{1}{2}(1-\Delta) u +\frac{1}{2}\lambda \be e^{\be u}=  \xi,
\end{equation}
where $\be, \lambda \in \mathbb{R} \backslash \{0\}$ and they proved that the local well-posedness depends again sensitively on the value of $\be^2>0$ as well as the sign of $\lambda$.
To end with, resorting to a paracontrolled approach, Oh and Okamoto \cite{oh-okamoto-2} worked on a stochastic heat equation with a quadratic nonlinearity but forced by a more irregular noise:
\begin{equation}\label{frac}
\partial_t u + (1-\Delta)u +u^2=\langle\nabla\rangle^{\alpha} \xi  ,
\end{equation}
where $\langle\nabla\rangle^{\alpha} \xi$ denotes a $\alpha$-derivative (in space) of a (Gaussian) space-time white noise on $\mathbb{T}^2\times \mathbb{R^+}$.
\smallskip
Our objective is to go one step further that is to study a stochastic heat equation on $\mathbb{R}^d$ (instead of $\mathbb{T}^d$) driven by the derivative (in space {\it and in time}) of a space-time fractional Brownian motion for $d\geq 1$.
Helped with the regularising effect of the heat semigroup which allows a gain of two derivatives in the fractional Sobolev setting, we will be able to derive existence results about (\ref{dim-d}). 

Our results can be summed up in the following way: 
 
\begin{theorem} 
Suppose that $d\geq 1$ and set $\alpha_H:=\big(2H_0+\sum_{i=1}^{d}H_i\big)-d$.
The picture below holds true:

\smallskip

\noindent
$(i)$ \textbf{Case $\alpha_H >0$.} The equation~\eqref{dim-d} is almost surely locally well-posed in $\mathcal{W}^{\beta,p}(\R^d)$ for some $\beta>0$ and $p\geq2$. 

\smallskip

\noindent
$(ii)$ \textbf{Case $\alpha_H \leq 0$.}  If $\alpha_H >-\frac{1}{4},$   then the equation~\eqref{dim-d}   can be interpreted in the Wick (renormalized) meaning and  it is    almost surely locally well-posed in $\mathcal{W}^{-\beta,p}(\R^d)$ for some $\beta>0$ and $p\geq2$. 
\end{theorem}

Let us now come back to the analysis of equation (\ref{dim-d}). We first rewrite it under the mild formulation, that is the (formal) equation
\begin{equation}
    u_t=e^{t\Delta} \phi +\int_0^t e^{(t-\tau)\Delta}(\rho^2 u_{\tau}^2)d\tau +\<Psi>_t
\end{equation}
where $e^{t\Delta}$ stands for the heat semigroup, and (morally) $\<Psi>_t:=\int_0^t e^{(t-\tau)\Delta}(\dot B_{\tau})d\tau.$ Here, we adopt Hairer's convention (see \cite{hai-14}) to denote the stochastic terms by trees; the vertex $\<circle>$ in $\<Psi>$ represents the random noise $\dot B$ and the edge corresponds to the Duhamel integral operator:
$$\mathcal{I}=(\partial_t -\Delta)^{-1}.$$
Now, in order to isolate the more irregular term $\<Psi>$, we resort to the so-called Da Prato-Debussche trick to rewrite the equation under the form:
\begin{multline}
v_t=e^{t\Delta}\phi+\int_0^t e^{(t-\tau)\Delta}(\rho^2 v_{\tau}^2)d\tau+2\int_0^t e^{(t-\tau)\Delta}((\rho v_{\tau})\cdot(\rho \<Psi>_{\tau}))\, d\tau\\
+\int_0^t e^{(t-\tau)\Delta}(\rho^2 \<Psi>_{\tau}^2)\, d\tau\, , \quad t\in [0,T] \, ,
\end{multline}
where $v:=u-\<Psi>$.
Our first and main objective is to define properly the solution $\<Psi>$ related to the linear equation:
\begin{equation}\label{linear}
 \left\{
    \begin{array}{ll}
  \partial_t\<Psi>-\Delta \<Psi> = \dot B, \hskip 0.3 cm t \in [0,T], \hspace{0,2cm} x \in \mathbb{R}^d,   \\
 \  \<Psi>(0,.)=0.
    \end{array}
\right.
\end{equation}
Our strategy is to consider $(\dot B_n)_{n  \geq 0}$ a smooth approximation of $\dot B$ and to show that the associated solutions $(\<Psi>_n)_{n  \geq 0}$ verifying
\begin{equation}
 \left\{
    \begin{array}{ll}
  \partial_t\<Psi>_n-\Delta \<Psi>_n = \dot B_n\, , \hskip 0.3 cm t \in [0,T]\, , \hspace{0,2cm} x \in \mathbb{R}^d\, ,   \\
 \  \<Psi>_n(0,.)=0\, ,
    \end{array}
\right.    
\end{equation}
form a Cauchy sequence in a convenient subspace (see Section \ref{stochastic} for more details). This leads us to the construction of the first order stochastic process $\<Psi>$.

\begin{proposition}\label{sto}
Suppose that $d\geq 1$. For all $(H_0,H_1,...,H_d)\in (0,1)^{d+1}$ and every  test (i.e., smooth compactly-supported) function $\chi:\mathbb{R}^d$ $\rightarrow$ $\mathbb{R}$, almost surely the sequence $(\chi \<Psi>_n)_{n  \geq 0}$ converges in the space $\mathcal{C}([0,T]; \mathcal{W}^{-\al,p}(\mathbb{R}^d))$ as soon as $2\leq p\leq \infty$ and
\begin{equation}\label{assump-gene-al}
\alpha > d-\bigg(2H_0+\sum_{i=1}^{d}H_i\bigg)=-\al_H\, .
\end{equation}
The limit of this sequence will be denoted by $\chi \<Psi>$.
\end{proposition}

It is now clear that the sign of $\al_H$ will play a major role in the resolution of ($\ref{dim-d}$). Indeed, when $\al_H >0$, $\chi \<Psi>(t)$ will be seen as a function defined on $\mathbb{R}^d$ whereas, when $\al_H<0$, $\chi \<Psi>(t)$ will be seen as a spatial distribution. It is important to remark that the precise nature of $\<Psi>$ is known up to multiplication by $\chi \in \cac_c^\infty(\R^d)$. To put it differently, the stochastic process $\<Psi>(t)$ is (almost surely) only a distribution on $\mathbb{R}^d$ but we have the more refined result : $\chi \<Psi> \in \mathcal{C}([0,T]; \mathcal{W}^{-\al,p}(\mathbb{R}^d))$ (see section \ref{subsec:glob-def} for more details). In the study of our equation, we will thus take $\chi=\rho$.

\smallskip

\subsubsection{The regular case}

\

\smallskip

In the following, we will call \emph{regular case} the situation where 
\begin{equation}\label{cond-regu-case}\tag{\textbf{H1}}
2H_0+\sum_{i=1}^{d}H_i >d \, .
\end{equation}
When this hypothesis is realized, $\al <0$ will be picked such that condition~\eqref{assump-gene-al} is satisfied, and therefore, resorting to Proposition~\ref{sto}, every element $\chi \<Psi>$ ($\chi\in \cac^\infty_c(\R^d)$) will be considered as a function of both time and space variables (with probability one) and consequently $\rho^2\<Psi>^2(t)$ will perfectly make sense as a classical function defined on $\mathbb{R}^d$. We thus propose the following natural interpretation of the model:

\begin{definition}\label{defi:sol-regu}
Let $d\geq 1$ and suppose that condition~\eqref{cond-regu-case} is verified. Remember that for all test function $\chi:\R^d \to \R$, $\chi\<Psi>$ is the process provided by Proposition~\ref{sto}.

\smallskip

A stochastic process $(u(t,x))_{t \in [0,T], x \in \mathbb{R}^d}$ is said to be a solution (on $[0,T]$) of equation~\eqref{dim-d} if, almost surely, the process $v:=u-\<Psi>$ is a solution of the mild equation
\begin{multline}
v_t=e^{t\Delta}\phi+\int_0^t e^{(t-\tau)\Delta}(\rho^2 v_{\tau}^2)d\tau+2\int_0^t e^{(t-\tau)\Delta}((\rho v_{\tau})\cdot(\rho \<Psi>_{\tau}))\, d\tau\\
+\int_0^t e^{(t-\tau)\Delta}(\rho^2 \<Psi>_{\tau}^2)\, d\tau\, , \quad t\in [0,T] \, .
\end{multline}
\end{definition}

\smallskip

We are now in a position to state our first existence result which will be a consequence of a standard fixed-point theorem:

\begin{theorem}[\textbf{Local well-posedness under~\eqref{cond-regu-case}}]\label{resu}
Suppose that $ d \geq 1$ and that condition~\eqref{cond-regu-case} is satisfied. Let $p\geq2$ and $\be$ be such that $0< \beta < 2H_0+\sum_{i=1}^{d}H_i-d$ and $\frac{d}{2p}<1+\frac{\be}{2}$. Finally, assume that $\phi \in \mathcal{W}^{\beta,p}(\mathbb{R}^d)$. In  this case, the assertions below hold true:

\smallskip

\noindent
$(i)$ Almost surely, there exists a time $T_0(\omega) >0$ such that equation~\eqref{dim-d} admits a unique solution~$u$ (in the meaning of Definition~\ref{defi:sol-regu}) in the subset 
$$\mathcal{S}_{T_0}:= \<Psi> + X^{\beta,p}(T_0),\hspace{0,5cm} \text{where} \ X^{\beta,p}(T_0):=\mathcal{C}([0,T_0]; \mathcal{W}^{\be,p}(\mathbb{R}^d)).$$

\smallskip

\noindent
$(ii)$ For any $n\geq 1$, let us note $u_n$ the \emph{smooth} solution of~\eqref{dim-d}, that is $u_n$ is the solution (in the meaning of Definition~\ref{defi:sol-regu}) related to $\rho \<Psi>_n$. Then, for all 
$$0< \beta < 2H_0+\sum_{i=1}^{d}H_i-d$$
and for any test function $\chi: \R^d \to \R$, the sequence $(\chi u_{n})_{n\geq 1}$ converges almost surely in $\mathcal{C}([0,T_0];\mathcal{W}^{\beta,p}(\mathbb{R}^d))$ to $\chi u$, where $u$ is the solution from item $(i)$.
\end{theorem}

\subsubsection{The rough case}

\

\smallskip

Let us now focus on the second situation, where 
\begin{equation}
2H_0+\sum_{i=1}^{d}H_i \leq d \, .
\end{equation}

\noindent
In this situation, $\rho \<Psi>(t)$ is only a distribution and not a function anymore. So the classical difficulty to properly define $\rho^2 \<Psi>^2(t)$ appears. A Wick-renormalization procedure permits to overcome this issue. To begin with, let us introduce the Wick-renormalized product 
\begin{equation}\label{defi:cerise}
\<Psi2>_n(t,x) := \<Psi>_n(t,x)^2-\sigma_n(t,x) \quad \text{where} \ \sigma_n(t,x):=\mathbb{E}\big[\<Psi>_n(t,x)^2\big] \, .
\end{equation}

\noindent
Before looking for a convenient subspace in which $(\rho^2 \<Psi2>_n)_{n\geq0}$ would be a Cauchy sequence, let us have a look at the renormalization constant.

\begin{proposition}\label{prop:renorm-cstt}
Let $d\geq 1$ and suppose that $2H_0+\sum_{i=1}^{d}H_i\leq d$. Then, the value of $\mathbb{E}\big[\<Psi>_n(t,x)^2\big] $ does not depend on $x$. Setting $\sigma_n(t):=\mathbb{E}\big[\<Psi>_n(t,x)^2\big]$, the asymptotic equivalence property below is obtained: for every $(t,x)\in (0,T]\times \R^d$,
\begin{equation}\label{estim-sigma-n}
    \sigma_n(t,x):=\mathbb{E}\big[\<Psi>_n(t,x)^2\big] \underset{n \rightarrow \infty}\sim  \left\{
    \begin{array}{ll}
		c_H^1\, n & \mbox{if} \ 2H_0+\sum_{i=1}^{d}H_i=d ,\\
        c_H^2\, 2^{2n(d-[2H_0+\sum_{i=1}^{d}H_i])} & \mbox{if} \ 2H_0+\sum_{i=1}^{d}H_i<d \, ,  
    \end{array}
\right.
\end{equation}
where $c_H^1,c_H^2>0$ are two constants.
\end{proposition}

\begin{remark}
It is interesting to note that the nature of the equivalent (that is linear when $ 2H_0+\sum_{i=1}^{d}H_i=d$ and geometric when $2H_0+\sum_{i=1}^{d}H_i<d $) is the same as in the wave setting (see \cite{deya-wave}) and in the Schr{\"o}dinger setting (see \cite{deya3}). Let us add that, in this case, the term in the right hand-side of (\ref{estim-sigma-n}) has no dependence on $t$. In fact, this results from a dissipative phenomenon of the heat equation (see the proof of Proposition \ref{equi}) contrary to an oscillating one in the case of the Schr{\"o}dinger equation (see \cite{deya3}). This kind of phenomenon is well-known in the parabolic setting and has already been observed in \cite{hoshino} for instance in its study of the KPZ equation with fractional derivatives of white noise.
\end{remark}

\noindent
Our definition of $\<Psi2>$, the second order stochastic term, is the following:
\

\begin{proposition}\label{sto1}
Fix $d\geq 1$ and $(H_0,H_1,...,H_d)\in (0,1)^{d+1}$ such that
\begin{equation}\label{cond-hurst-psi-2}\tag{\textbf{H2}}
d-\frac{1}{4}< 2H_0+\sum_{i=1}^{d}H_i \leq d\, .
\end{equation}
Then, for all test function $\chi:\mathbb{R}^d\rightarrow\mathbb{R}$, almost surely the sequence $(\chi^2\<Psi2>_n)_{n \geq 1}$ (defined by \eqref{defi:cerise}) converges in the space $\mathcal{C}([0,T]; \mathcal{W}^{-2\al,p}(\mathbb{R}^d))$ as soon as $2\leq p\leq \infty$ and $\al$ satisfying \eqref{assump-gene-al}.

\smallskip

The limit of this sequence will be denoted by $\chi^2 \<Psi2>$.
\end{proposition}

\noindent
With the above constructed stochastic processes, the following \emph{Wick interpretation} of the model naturally comes to mind:

\begin{definition}\label{defi:sol}
Let $d\geq 1$ and assume that condition~\eqref{cond-hurst-psi-2} is verified. Remember that for all test function $\chi:\R^d \to \R$, $\chi\<Psi>$ and $\chi^2\<Psi2>$ are the processes defined in Proposition~\ref{sto} and Proposition~\ref{sto1}, respectively.

\smallskip
In this setting, a stochastic process $(u(t,x))_{t \in [0,T], x \in \mathbb{R}^d}$ is said to be a \emph{Wick} solution (on $[0,T]$) of equation~\eqref{dim-d} if, almost surely, the process $v:=u-\<Psi>$ is a solution of the mild equation
\begin{multline}\label{equa-v-rough}
v_t=e^{t\Delta}\phi+\int_0^t e^{(t-\tau)\Delta}(\rho^2 v_{\tau}^2)d\tau+2\int_0^t e^{(t-\tau)\Delta}((\rho v_{\tau})\cdot(\rho \<Psi>_{\tau}))\, d\tau\\
+\int_0^t e^{(t-\tau)\Delta}(\rho^2 \<Psi2>_{\tau})\, d\tau\, , \quad t\in [0,T] \, .
\end{multline}
\end{definition}

In this case, a major difficulty is the treatment of the term $(\rho v) \cdot (\rho \<Psi>)$ that will be understood as the
product of a distribution $\rho \<Psi>$ with a regular enough function $\rho v$. Indeed, should $\rho \<Psi>$ be of
Sobolev regularity $-\alpha$ and $\rho v$ be a function of Sobolev regularity $\beta$ with $\beta>\alpha$, $(\rho v) \cdot (\rho \<Psi>)$ can be defined as a distribution of order $-\alpha$ (see Section \ref{sec:irreg-case-det}  for more details). Again, the regularising effect of the heat semigroup will help us to find a stable subspace. The main result of this paper can be stated in the following way:

\begin{theorem}[\textbf{Local well-posedness under (H2)}]\label{resu1}
Suppose that $d\geq 1$ and $p\geq2$ verifies that $p\geq \frac{2d}{3}.$
Assume that $d-\frac{1}{4}<2H_0+\sum_{i=1}^{d}H_i\leq d$.
Fix $\al>0$ such that
\begin{equation}\label{alpha1}
d-\bigg(2H_0+\sum_{i=1}^{d}H_i\bigg)<\al<\frac{1}{4}.
\end{equation}
 Then the assertions below hold true:

\smallskip

\noindent
$(i)$ One can find $\be>0$ such that almost surely, for every $\phi \in \mathcal{W}^{-\alpha,p}(\mathbb{R}^d)$, there exists a time $T_0(\omega)>0$ for which equation~\eqref{dim-d} admits a unique Wick solution $u$ (in the meaning of Definition~\ref{defi:sol}) in the set
$$\mathcal{S}_{T_0}:= \<Psi> + X^{\alpha,\be}(T_0)\, ,$$
where
$$X^{\alpha,\be}(T_0):=\mathcal{C}([0,T_0]; \mathcal{W}^{-\alpha,p}(\mathbb{R}^d))\cap \mathcal{C}((0,T_0]; \mathcal{W}^{\be,p}(\mathbb{R}^d)) .$$

\smallskip

\noindent
$(ii)$ For any $n\geq 1$, let us note $\widetilde{u}_n$ the \emph{smooth} Wick solution of~\eqref{dim-d}, that is $\widetilde{u}_n$ is the solution (in the meaning of Definition~\ref{defi:sol}) related to the pair $(\rho \<Psi>_n,\rho^2\<Psi2>_n)$. Then, for all $\al$ satisfying~(\ref{alpha1}) and for any test function $\chi:\R^d \to \R$, the sequence $(\chi\widetilde{u}_{n})_{n\geq 1}$ converges almost surely in $\mathcal{C}([0,T_0];\mathcal{W}^{-\al,p}(\mathbb{R}^d))$ to $\chi u$, where $u$ is the Wick solution from item $(i)$.
\end{theorem}

\

\begin{remark}
Thanks to the regularising effect of the heat semigroup,  we are able to solve equation (\ref{dim-d}) for all dimension $d\geq1$ as soon as $\al<\frac{1}{4}$. This situation highly contrasts with the Schr{\"o}dinger setting (see \cite{deya3}) where the restriction comes from the deterministic part of the equation. Nonetheless, although the restriction in Theorem \ref{resu1} comes from the stochastic constructions, we will be able to improve this result in dimension $d=2.$ See subsection \ref{irreg} below.
\end{remark}

\begin{remark}
Let us go back to the role of the cut-off function $\rho$ in equation (\ref{dim-d}). Contrary to the Schr{\"o}dinger setting where $\rho$ played an essential role to obtain local regularity, in our case, $\rho$ is only needed to measure precisely the regularity of the stochastic processes $\rho \<Psi>(t,.),$ $\rho^2 \<Psi2>(t,.)$ in some $\mathcal{W}^{\al,p}(\mathbb{R}^d)$ spaces.
\end{remark}

\subsubsection{Further study of the model when d=2}\label{irreg}

\

\smallskip

At this point of the analysis, one may wonder whether the restriction $2 H_{0}+\sum_{i=1}^{d} H_{i}>d-\frac{1}{4}$ in \eqref{cond-hurst-psi-2} (that comes from our technical computations) is optimal. The following proposition provides us with a first partial result in this direction.

\begin{proposition}\label{prop:limite-case}
Fix $d\geq1$. Let $(H_0,H_1,...,H_d)\in (0,1)^{d+1}$ be such that 
\begin{equation}\label{constrainte-h-i-opt}
2H_0+H_1+\cdots+H_d \leq \frac{3}{4}d
\end{equation}
and let $\chi:\R^d\to \R$ be a non-zero, smooth and compactly-supported function. Then, for all $\al >0$ and $t>0$, one has $\mathbb{E}\Big[ \|\chi\cdot \<Psi2>_n(t,.)\|_{H^{-2\al}(\R^d)}^2\Big] \to \infty$ as $n\to \infty$.
\end{proposition}

Based on the latter explosion phenomenon for $\<Psi2>_n$, we can at least conclude that the condition \eqref{cond-hurst-psi-2} in Proposition \ref{sto1} and Theorem \ref{resu1} is optimal when $d=1$.

When $d \geq 2$, one can observe that the situation where $\frac{3}{4} d<2 H_{0}+\sum_{i=1}^{d} H_{i} \leq d-\frac{1}{4}$ is not covered by the above results. In fact, we must admit that we are not able, for the moment, to offer a complete picture of the situation, that is a general well-posedness result for every $d \geq 2$ and for all indexes such that $2 H_{0}+\sum_{i=1}^{d} H_{i}>\frac{3}{4}d.$

However, we propose to make a first step toward this ambitious challenge by going further with the analysis in the two-dimensional setting $d=2.$ In this case, it turns out that several of our arguments and estimates behind the construction of $\<Psi2>$ can be slightly refined, leading us to the following statement:

\begin{proposition}\label{prop:stoch-constr}
Let $(H_0,H_1,H_2)\in (0,1)^{3}$ be such that 
\begin{equation}\label{constraint-h-i}
0<H_1<\frac34 \quad , \quad 0<H_2< \frac34 \quad , \quad \frac{3}{2} < 2H_0+H_1+H_2 \leq  \frac74 \ .
\end{equation}
Then, for all smooth compactly-supported function $\chi:\R^2\to \R$, $T>0$, $2\leq p\leq \infty$ and 
\begin{equation}\label{regu-2}
\al >2- (2H_0+H_1+H_2)\ ,
\end{equation}
 the sequence $(\chi^2\<Psi2>_n)_{n \geq 1}$ (defined by \eqref{defi:cerise}) converges almost surely in the space $\mathcal{C}([0,T]; \mathcal{W}^{-2\al,p}(\mathbb{R}^2))$.
\smallskip

We still denote the limit of this sequence by  $\chi^2 \<Psi2>$.
\end{proposition}

\noindent
According to Proposition \ref{prop:limite-case} with $d=2$, this construction is optimal. We are now in a position to state our more general result when $d=2$:

\begin{theorem}[\textbf{Roughest case}]\label{resu2}
Suppose that $d=2$ and $p\geq2.$
Let $(H_0,H_1,H_2)\in (0,1)^{3}$ be such that 
\begin{equation}
0<H_1<\frac34 \quad , \quad 0<H_2< \frac34 \quad , \quad \frac{3}{2} < 2H_0+H_1+H_2 \leq  \frac74 \ .
\end{equation}
Fix $\al>0$ such that 
\begin{equation}\label{alpha2}
2-(2H_0+H_1+H_2)<\al<\frac{1}{2}.
\end{equation}
Then the assertions below hold true:

\smallskip

\noindent
$(i)$ One can find $\be>0$ such that almost surely, for every $\phi \in \mathcal{W}^{-\alpha,p}(\mathbb{R}^2)$, there exists a time $T_0(\omega)>0$ for which equation~\eqref{dim-d} admits a unique Wick solution $u$ (in the meaning of Definition~\ref{defi:sol}) in the set
$$\mathcal{S}_{T_0}:= \<Psi> + X^{\alpha,\be}(T_0)\, ,$$
where
$$X^{\alpha,\be}(T_0):=\mathcal{C}([0,T_0]; \mathcal{W}^{-\alpha,p}(\mathbb{R}^2))\cap \mathcal{C}((0,T_0]; \mathcal{W}^{\be,p}(\mathbb{R}^2)) .$$

\smallskip

\noindent
$(ii)$ For any $n\geq 1$, let us note $\widetilde{u}_n$ the \emph{smooth} Wick solution of~\eqref{dim-d}, that is $\widetilde{u}_n$ is the solution (in the meaning of Definition~\ref{defi:sol}) related to the pair $(\rho \<Psi>_n,\rho^2\<Psi2>_n)$. Then, for all $\al$ satisfying~(\ref{alpha2}) and for any test function $\chi:\R^2 \to \R$, the sequence $(\chi\widetilde{u}_{n})_{n\geq 1}$ converges almost surely in $\mathcal{C}([0,T_0];\mathcal{W}^{-\al,p}(\mathbb{R}^2))$ to $\chi u$, where $u$ is the Wick solution from item $(i)$.
\end{theorem}

\begin{remark}\label{rk:limitations}
Again, the restriction comes from the stochastic part of the study (see Proposition~\ref{prop:stoch-constr}) that tackles with the existence of the stochastic element $\<Psi2>$. This situation is analogous to those in the wave setting (see for example ~\cite{deya-wave,gubi-koch-oh,gubi-koch-oh-2,oh-okamoto-2}), where the strategy is limited by the domain of validity of the stochastic constructions.
\end{remark}

The above well-posedness theorem can be viewed as the main novelty of our work. We can sum up our results in the following way ; equation (\ref{dim-d}) is well-posed in some $\mathcal{W}^{\al,p}(\mathbb{R}^2)$ space when $$\frac{3}{2} < 2H_0+H_1+H_2$$
and is ill-posed as soon as 
$$2H_0+H_1+H_2\leq \frac{3}{2}$$
in the sense that the second order stochastic term is not a continuous function of time with values in Sobolev spaces.
These results can be viewed as a generalization of those obtained by Oh and Okamoto in \cite{oh-okamoto-2} where the parameter $\alpha$ (related to the derivatives of their Gaussian noise) has to satisfy analogous constraints to those of the combination $2H_0+H_1+H_2$.

\begin{remark}\label{rk:compari-lin}
It is interesting to compare the above regularity restriction~\eqref{assump-gene-al} for $\<Psi>$ in Proposition \ref{sto} with the analogous results of~\cite{deya-wave} in the wave setting and of~\cite{deya3} in the Schr{\"o}dinger setting, that is when replacing $\partial_t\<Psi>-\Delta \<Psi>$ with $\partial^2_t\<Psi>-\Delta \<Psi>$ (respectively $\imath\partial_t\<Psi>-\Delta \<Psi>$) . In these situations, the conditions on $\al$ are
$$
\alpha_{\text{wave}} > d-\frac12-\sum_{i=0}^{d}H_i \quad {and} \quad \alpha_{\text{schr{\"od}}} > d+1-\bigg(2H_0+\sum_{i=1}^{d}H_i\bigg) .
$$
In fact, the method presented in this paper would allow us to define the linear solution $\<Psi>$ of several other fractional problems. Let us briefly describe some results that could be obtained along the same procedure (see Section \ref{stochastic}).
First of all, for fractional Schr{\"o}dinger-type  equations of the form:
$$i \partial_{t} \<Psi>(t, x)+|\nabla|^{\sigma} \<Psi>(t, x)=0,$$
where $|\nabla|^{\sigma}$ is the Fourier multiplier by $|\xi|^{\sigma}$ and $\sigma \in (0,+\infty)$, that generalize the classical one (take $\sigma=2$), one could prove that the associated condition on $\al$ becomes 
$$ \alpha_{\text{ frac schr{\"od}}} > d+\frac{\sigma}{2}-\bigg(\sigma H_0+\sum_{i=1}^{d}H_i\bigg) .
$$
Likewise, we could construct the linear solution $\<Psi>$ for fractional wave-type  equations of the form:
$$ \partial^2_{t} \<Psi>(t, x)+|\nabla|^{2\sigma} \<Psi>(t, x)=0,$$
that generalize the standard one (take $\sigma=1$). In this case, the associated condition on $\al$ would be 
$$ \alpha_{\text{ frac wave}} > d-\frac{\sigma}{2}-\bigg(\sigma H_0+\sum_{i=1}^{d}H_i\bigg) .
$$
To end with, it is possible to construct the first order stochastic process $\<Psi>$ for fractional heat-type  equations of the form:
$$ \partial_{t} \<Psi>(t, x)+|\nabla|^{\sigma} \<Psi>(t, x)=0,$$
as soon as 
$$ \alpha_{\text{ frac heat}} > d-\bigg(\sigma H_0+\sum_{i=1}^{d}H_i\bigg) .$$
These considerations allow us to understand that the key parameter in the construction of the linear solution is the exponent of the Fourier multiplier that  appears in front of $H_0$ (that described the irregularity of $B$ with respect to the time $t$) whereas the combination $\sum_{i=1}^{d}H_i$ stays unchanged.

\end{remark}

\subsection{Notations}\label{subsec:notations}

 Let $d\geq 1$. Throughout the paper, a test function (on $\R^d$) will be any function $\rho:\R^d \to \R$ that is smooth and compactly-supported. Naturally, let us denote by $\mathcal{S}(\R^d)$ the space of Schwartz functions on $\R^d$.

\smallskip

The main functional spaces of this article are the Sobolev spaces defined for all $s\in \mathbb{R}$ and $1\leq p\leq \infty$ as
$$\mathcal{W}^{s,p}(\mathbb{R}^d):=\left\{f\in \mathcal{S}'(\mathbb{R}^d):\ \|f \|_{\mathcal{W}^{s,p}}=\|\mathcal{F}^{-1}(\{1+|.|^2\}^{\frac{s}{2}}\mathcal{F}f) |L^p(\mathbb{R}^d)\| 
  <\infty \right\} \, ,$$
where the Fourier transform $\cf$, resp. the inverse Fourier transform $\cf^{-1}$, is defined with the following convention: for every $f \in \mathcal{S}(\mathbb{R}^d)$ and $x \in \mathbb{R}^d,$
$$\mathcal{F}(f)(x)=\hat{f}(x):=\int_{\mathbb{R}^d} f(y)e^{-\imath \langle x, y \rangle}dy\, , \quad \text{resp.} \ \ \mathcal{F}^{-1}(f)(x):=\frac{1}{(2\pi)^d}\int_{\mathbb{R}^d} f(y)e^{\imath \langle x, y \rangle}dy\, .$$
We also set $H^s(\mathbb{R}^d):=\mathcal{W}^{s,2}(\mathbb{R}^d)$.

\smallskip

Now, as soon as space-time functions (or distributions) are concerned, we may use the following shortcut notation: for every $T\geq 0$, $1\leq p,q\leq \infty$ and $s\in \R$, 
\begin{equation}\label{shortcut-spaces}
L^p_T \cw^{s,q}:=L^p([0,T];\mathcal{W}^{s,q}(\mathbb{R}^d))\, , \quad  \quad \|.\|_{L^p_T \cw^{s,q}}:=\|.\|_{L^p([0,T];\mathcal{W}^{s,q}(\mathbb{R}^d))} \, .
\end{equation}
The notation $\mathcal{C}([0,T];\mathcal{W}^{s,q}(\mathbb{R}^d))$ will refer to the set of continuous functions on $[0,T]$ taking values in $\mathcal{W}^{s,q}(\mathbb{R}^d)$.

\smallskip

\subsection{Outline of the study}

\

\smallskip

The rest of the article is organized as follows. In Section~\ref{sec:regular}, we prove the local well-posedness result in the regular case. Section~\ref{sec:irreg-case-det} is devoted to the analysis in the irregular case. Then, in Section~\ref{stochastic} and Section \ref{rough}, we develop the stochastic constructions at the core of the equation. To end with, in Section \ref{end}, we combined the results from the previous sections to prove Theorem \ref{resu}, Theorem \ref{resu1} and Theorem \ref{resu2}.
\medskip

\textit{Throughout the paper, the notation $A\lesssim B$ will be used to signify that there exists an irrelevant constant $c$ such that $A\leq c B$.}

\section{Analysis of the deterministic equation under condition \textbf{(H1)}}\label{sec:regular}

The aim of this section is to deal with the model in the \emph{regular} case, that is when assumption~\eqref{cond-regu-case} on the Hurst index is verified, and the linear solution $\rho \<Psi>$ (defined by Proposition~\ref{sto}) is a function of both time \emph{and} space. Recall that in this situation, the equation is interpreted through Definition~\ref{defi:sol-regu}. Thus, what we would like to solve in this section is the equation 
\begin{multline}\label{det-0}
v_t=e^{t\Delta}\phi+\int_0^t e^{(t-\tau)\Delta}(\rho^2 v_{\tau}^2)d\tau+2\int_0^t e^{(t-\tau)\Delta}(\rho v_{\tau}\cdot\rho \<Psi>_{\tau})\, d\tau\\
+\int_0^t e^{(t-\tau)\Delta}(\rho^2 \<Psi>_{\tau}^2)\, d\tau\, , \quad t\in [0,T] \, .
\end{multline}

\smallskip

Contrary to the next sections where several probabilistic arguments will be used, our strategy towards a (local) solution $v$ for~\eqref{det-0} will be based on deterministic estimates only. To put it differently, we henceforth consider $\rho\<Psi>$ as a fixed (i.e., non-random) element in the space $\mathcal{C}([0,T]; \mathcal{W}^{\be,p}(\mathbb{R}^d))$
for some appropriate $0<\beta<1$ (where $\beta=-\alpha$ is given by Proposition~\ref{sto}) and $p\geq2$, and try to solve the deterministic equation below: for $\luxor \in \mathcal{C}([0,T]; \mathcal{W}^{\be,p}(\mathbb{R}^d))$,
\begin{multline}\label{det}
v_t=e^{t\Delta}\phi+\int_0^t e^{(t-\tau)\Delta}(\rho^2 v_{\tau}^2)d\tau+2\int_0^t e^{(t-\tau)\Delta}(\rho v_{\tau}\cdot\luxor_{\tau})\, d\tau\\
+\int_0^t e^{(t-\tau)\Delta}(\luxor_{\tau}^2)\, d\tau\, , \quad t\in [0,T] \, .
\end{multline}
Let us present the tools related to the two main operations in~\eqref{det}.

\subsection{$\mathcal{W}^{\al,p}-\mathcal{W}^{\al,q}$ estimate}

\

\smallskip

We recall the well-known integrability property of the heat semigroup resulting from the Riesz-Thorin theorem.  

\begin{proposition}\label{estimate}
Let $1\leq q \leq p \leq \infty$, $\al \in \mathbb{R}$ and $u_0 \in \mathcal{W}^{\al,q}(\mathbb{R}^d).$
Then for all $t>0$,
$$\|e^{t\Delta}u_0\|_{\mathcal{W}^{\al,p}}\leq (4\pi t)^{-\frac{d}{2}(\frac{1}{q}-\frac{1}{p})} \|u_0\|_{\mathcal{W}^{\al,q}}.$$
\end{proposition}

\subsection{Pointwise multiplication}

\

\smallskip

The second tool to deal with equation (\ref{det}) is a refined fractional Leibniz rule that can be found in \cite{grafakos}. 
\begin{lemma}[Kato-Ponce inequality]\label{lem:frac-leibniz}
Let $1\leq r<\infty$ and $1<p_1,p_2,q_1,q_2< \infty$ verifying 
$$\frac{1}{r}=\frac{1}{p_1}+\frac{1}{p_2}=\frac{1}{q_1}+\frac{1}{q_2} \, .$$ 
Let $s\geq 0$.
Then one has
$$\|u\cdot v\|_{\mathcal{W}^{s,r}(\mathbb{R}^d)}\lesssim \|u\|_{\mathcal{W}^{s,p_1}(\mathbb{R}^d)}\|v\|_{L^{p_2}(\mathbb{R}^d)}+\|u\|_{L^{q_1}(\mathbb{R}^d)}\|v\|_{\mathcal{W}^{s,q_2}(\mathbb{R}^d)} \, .$$
\end{lemma}

\subsection{About the resolution of the deterministic equation}\label{subsec:determ-regu}

\

\smallskip

Our main result regarding equation~\eqref{det} can be formulated in the following way:
\begin{theorem}\label{thm:regular}
Fix $\beta \in (0,1).$ Assume that $d\geq 1$ and $p\geq2$ is such that $\frac{d}{2p}<1+\frac{\be}{2}.$ For every $T>0$, define the space $X^{\beta,p}(T)$ as
\begin{equation}\label{scales-spaces-x}
X^{\be,p}(T):=\mathcal{C}([0,T]; \mathcal{W}^{\be,p}(\mathbb{R}^d)) \, ,
\end{equation}
equipped with the norm
$$\|v\|_{X(T)}:=\|v\|_{X^{\be,p}(T)}=\|v\|_{L^{\infty}_T \mathcal{W}^{\be,p}}.$$
Then for every $\phi \in \mathcal{W}^{\beta,p}(\mathbb{R}^d),$ one can find a time $T_0>0$ such that, for every $\luxor \in X^{\beta,p}(T_0),$ equation~\eqref{det} admits a unique solution in $X^{\beta,p}(T_0)$.
\end{theorem}

\smallskip

This local well-posedness result will be the consequence of a classical fixed-point argument. In order to implement this procedure, let us introduce the map $\Gamma$ defined by the right-hand side of~\eqref{det}, that is: for all $\phi \in \mathcal{W}^{\beta,p}(\mathbb{R}^d)$, $v, \luxor \in X^{\beta,p}(T)$, $T\geq 0$, set
\begin{multline*}
\Gamma_{T, \luxor}(v)_{t}:=e^{t\Delta}\phi+\int_0^t e^{(t-\tau)\Delta}(\rho^2 v_{\tau}^2)d\tau+2\int_0^t e^{(t-\tau)\Delta}(\rho v_{\tau}\cdot\luxor_{\tau})\, d\tau\\
+\int_0^t e^{(t-\tau)\Delta}(\luxor_{\tau}^2)\, d\tau\, , \quad t\in [0,T] \, .
\end{multline*}

\begin{proposition}\label{control-regular}
In the setting of Theorem~\ref{thm:regular}, the bounds below hold true: one can find $\varepsilon>0$ such that for all $0\leq T \leq 1$, $\phi \in \mathcal{W}^{\beta,p}(\mathbb{R}^d), \luxor_1,\luxor_2, v, v_1, v_2 \in X^{\beta,p}(T):=X(T)$, 
\begin{equation}\label{bound1}
\|\Gamma_{T, \luxor_1}(v)\|_{X(T)}\lesssim \|\phi\|_{\mathcal{W}^{\beta,p}(\mathbb{R}^d)} +T^{\varepsilon}\Big[\|v\|_{X(T)}^2+\| \luxor_1\|_{X(T)}\|v\|_{X(T)}+\|\luxor_1\|_{X(T)}^2\Big] \, ,
\end{equation}
and
\begin{align}
&\|\Gamma_{T, \luxor_1}(v_1)-\Gamma_{T, \luxor_2}(v_2)\|_{X(T)}\nonumber \\
&\lesssim T^{\varepsilon}\Big[\|v_1-v_2\|_{X(T)}\big\{\|v_1\|_{X(T)}+\| v_2\|_{X(T)}\big\}+\|\luxor_1-\luxor_2\|_{X(T)}\| v_1\|_{X(T)}\nonumber\\
&\hspace{1cm}+\|\luxor_2\|_{X(T)}\|v_1-v_2\|_{X(T)}+\|\luxor_1-\luxor_2\|_{X(T)}\big\{\|\luxor_1\|_{X(T)}+\|\luxor_2\|_{X(T)}\big\}\Big]\, ,\label{bound2}   
\end{align}
where the proportional constants only depend on $\be$, $p$ and $\rho$.
\end{proposition}

Before we focus on the proof of this proposition, let us briefly remember that, once endowed~\eqref{bound1}-\eqref{bound2}, the result of Theorem~\ref{thm:regular} follows from a classical application of the Picard fixed-point theorem. Precisely, using~\eqref{bound1}, we can first show that for all $T=T(\phi,\luxor)>0$ small enough, there exists a ball in $X(T)$ that is stable through the action of $\Gamma_{T, \luxor}$. Then, helped with~\eqref{bound2} (used with $\luxor_1=\luxor_2=\luxor$), the fact that $\Gamma_{T, \luxor}$ is actually a contraction on this ball is easily deduced (for $T>0$ possibly even smaller), which completes the proof of the theorem.

\smallskip

Let us remark that the continuity of $\Gamma_{T, \luxor}$ with respect to $\luxor$ (a direct consequence of~\eqref{bound2}) will be a major ingredient toward item $(ii)$ of Theorem~\ref{resu}.

\smallskip
We begin by establishing some estimates.
\begin{lemma}\label{techlemm}
Fix $\beta \in (0,1).$ Assume that $d\geq 1$ and $p\geq2$ is such that $\frac{d}{2p}<1+\frac{\be}{2}.$ There exists $\varepsilon>0$ such that for all $0\leq T\leq 1,$ $f_i=\rho v$ or $f_i=\rho\<Psi>,$
$$\Big\|\int_0^t e^{(t-\tau)\Delta}(f_1 \cdot f_2(\tau))d\tau\Big\|_{X^{\be,p}(T)}\lesssim T^{\varepsilon}\|f_1\|_{X^{\be,p}(T)}\|f_2\|_{X^{\be,p}(T)}.$$
\end{lemma}

\begin{proof}
Let $0\leq t \leq T, 1\leq q<p$ and $1\leq r <\infty$ be such that $\frac{1}{q}=\frac{1}{p}+\frac{1}{r}.$
Using successively Minkowsky inequality, Proposition \ref{estimate} and Lemma \ref{lem:frac-leibniz}, we get

\begin{eqnarray*}
\Big\|\int_0^t e^{(t-\tau)\Delta}(f_1 \cdot f_2(\tau))d\tau\Big\|_{\mathcal{W}^{\be,p}}&\leq& \int_0^t \|e^{(t-\tau)\Delta}(f_1 \cdot f_2(\tau))\|_{\mathcal{W}^{\be,p}}d\tau\\
&\lesssim&\int_0^t (t-\tau)^{-\frac{d}{2}(\frac{1}{q}-\frac{1}{p})}\|f_1 \cdot f_2(\tau)\|_{\mathcal{W}^{\be,q}}d\tau \\
&\lesssim&\int_0^t(t-\tau)^{-\frac{d}{2}(\frac{1}{q}-\frac{1}{p})}[\|f_1(\tau)\|_{\mathcal{W}^{\be,p}}\|f_2(\tau)\|_{L^r}+\|f_2(\tau)\|_{\mathcal{W}^{\be,p}}\|f_1(\tau)\|_{L^r}]d\tau.
\end{eqnarray*}
Now, as soon as $\frac{1}{q}>\max\big(\frac{2}{p}-\frac{\be}{d},\frac{1}{p})$, observe that $\be \geq d(\frac1p-\frac1r)$ and consequently we have the Sobolev embedding $
\mathcal{W}^{\be,p}(\mathbb{R}^d)\hookrightarrow L^{r}(\mathbb{R}^d)$ that leads to
\begin{eqnarray*}
\Big\|\int_0^t e^{(t-\tau)\Delta}(f_1 \cdot f_2(\tau))d\tau\Big\|_{\mathcal{W}^{\be,p}}&\lesssim&\int_0^t(t-\tau)^{-\frac{d}{2r}}\|f_1(\tau)\|_{\mathcal{W}^{\be,p}}\|f_2(\tau)\|_{\mathcal{W}^{\be,p}}d\tau\\
&\lesssim& T^{1-\frac{d}{2r}}\|f_1\|_{X^{\be,p}(T)}\|f_2\|_{X^{\be,p}(T)}\\
&\lesssim& T^{1-\frac{d}{2p}+\frac{\be}{2}}\|f_1\|_{X^{\be,p}(T)}\|f_2\|_{X^{\be,p}(T)}.
\end{eqnarray*}
Taking the supremum over $t \in [0,T],$ we obtain the desired conclusion:
$$\Big\|\int_0^t e^{(t-\tau)\Delta}(f_1 \cdot f_2(\tau))d\tau\Big\|_{X^{\be,p}(T)}\lesssim T^{1-\frac{d}{2p}+\frac{\be}{2}}\|f_1\|_{X^{\be,p}(T)}\|f_2\|_{X^{\be,p}(T)}.$$
\end{proof}

\smallskip

\begin{proof}[Proof of Proposition~\ref{control-regular}]
Let us bound each of the four terms in the expression of $\Gamma_{T, \luxor}$:
\begin{multline*}
\Gamma_{T, \luxor}(v)_{t}:=e^{t\Delta}\phi+\int_0^t e^{(t-\tau)\Delta}(\rho^2 v_{\tau}^2)d\tau+2\int_0^t e^{(t-\tau)\Delta}(\rho v_{\tau}\cdot\luxor_{\tau})\, d\tau\\
+\int_0^t e^{(t-\tau)\Delta}(\luxor_{\tau}^2)\, d\tau\, , \quad t\in [0,T] \, .
\end{multline*}
On the one hand, for all $0\leq t \leq T$, using the expression of the heat semigroup, it holds that $\|e^{t\Delta}\phi\|_{\mathcal{W}^{\be,p}(\mathbb{R}^d)}\lesssim \|\phi\|_{\mathcal{W}^{\be,p}(\mathbb{R}^d)},$ yielding to
\begin{equation}\label{est1}
\|e^{t\Delta}\phi\|_{X^{\be,p}(T)}\lesssim \|\phi\|_{\mathcal{W}^{\be,p}(\mathbb{R}^d)}.
\end{equation}
On the other hand, the bound concerning the three other terms is a straight consequence from Lemma \ref{techlemm} and provides us with the bound (\ref{bound1}) we are looking for. The second one (\ref{bound2}) can be obtained with quite the same arguments.
\end{proof}

\section{Analysis of the deterministic equation under condition \textbf{(H2)}}\label{sec:irreg-case-det}

The objective of this section is to cope with the wellposedness issue in the rough situation, that is when condition~\eqref{cond-hurst-psi-2} on the Hurst indexes is verified. We recall that in this rough case, the model is understood in the meaning of Definition~\ref{defi:sol}, that is as
\begin{multline}\label{ccl-sto}
v_t=e^{t\Delta}\phi+\int_0^t e^{(t-\tau)\Delta}(\rho^2 v_{\tau}^2)d\tau+2\int_0^t e^{(t-\tau)\Delta}(\rho v_{\tau}\cdot\rho \<Psi>_{\tau})\, d\tau\\
+\int_0^t e^{(t-\tau)\Delta}(\rho^2 \<Psi2>_{\tau})\, d\tau\, , \quad t\in [0,T] \, .
\end{multline}
where the processes $\rho \<Psi>$ and $\rho^2 \<Psi2>$ are those defined in Proposition~\ref{sto} and Proposition~\ref{sto1}.

\smallskip

In order to handle~\eqref{ccl-sto}, we intend to follow the same deterministic strategy as in Section~\ref{sec:regular}. To put it differently, we will consider the pair $(\rho \<Psi>,\rho^2 \<Psi2>)$ as a fixed element in the subspace
\begin{equation}\label{defi-space-e-al}
\mathcal{R}_{\al,p}:=L^\infty\big([0,T];\mathcal{W}^{-\alpha,p}(\mathbb{R}^d)\big)\times L^\infty\big([0,T];\mathcal{W}^{-2\alpha,p}(\mathbb{R}^d)\big) \, ,
\end{equation}
where $0<\al<\frac{1}{4}$ (coming from Propositions~\ref{sto} and~\ref{sto1}), $p\geq 2$ and then the aim will be to solve the more general deterministic equation: for $(\luxor,\cherry) \in \mathcal{R}_{\alpha,p}$,

\begin{multline}\label{ccl}
v_t=e^{t\Delta}\phi+\int_0^t e^{(t-\tau)\Delta}(\rho^2 v_{\tau}^2)d\tau+2\int_0^t e^{(t-\tau)\Delta}(\rho v_{\tau}\cdot\luxor_{\tau})\, d\tau\\
+\int_0^t e^{(t-\tau)\Delta}(\cherry_{\tau})\, d\tau\, , \quad t\in [0,T] \, .
\end{multline}

The main originality of the situation (compared to Section~\ref{sec:regular}) is the lack of regularity of $\luxor_\tau$ and $\cherry_\tau$, that can only be seen as negative-order Sobolev terms (remark indeed that $\al >0$ in~\eqref{defi-space-e-al}).

\subsection{Regularising effect of the heat semigroup}

\

\smallskip

The use of the fractional Sobolev spaces allows us to quantify the regularising effect of the heat semigroup (see \cite{pazy}).

\begin{proposition}\label{heat}
Let $1<p<\infty$. For any $s_1<s_2$, there exists a constant $C(s_1,s_2)<\infty$ such that
$$\|e^{t\Delta}\phi\|_{\mathcal{W}^{s_2,p}(\mathbb{R}^d)}\leq C(s_1,s_2) (1+t^{-\frac{s_2-s_1}{2}})\|\phi\|_{\mathcal{W}^{s_1,p}(\mathbb{R}^d)}$$
holds for all $\phi \in \mathcal{W}^{s_1,p}(\mathbb{R}^d)$ and $t>0$.
\end{proposition}

\subsection{Pointwise multiplication}\label{subsec:point-mult-inter}

\

\smallskip

The solution we propose to deal with the product $(\rho v_{\tau})\cdot(\luxor_{\tau})$ in~\eqref{ccl} will be to resort to the following general multiplication property in Sobolev spaces (see e.g.~\cite[Section 4.4.3]{prod} for a proof of this result) which garantees that $(\rho v_{\tau})\cdot(\luxor_{\tau})$ exists and inherits the worst regularity. More precisely:

\begin{lemma}\label{lem:product}
Fix $d\geq 1$. Let $\al,\be >0$ and $1 \leq p,p_1,p_2< \infty$ be such that
$$\frac{1}{p}= \frac{1}{p_1}+\frac{1}{p_2} \quad \text{and} \quad 0<\al<\be  \, .$$
If $f\in \cw^{-\al,p_1}(\R^d)$ and $g\in \cw^{\be,p_2}(\R^d)$, then $f\cdot g\in \cw^{-\al,p}(\R^d)$ and
$$
\| f\cdot g\|_{\cw^{-\al,p}} \lesssim \|f\|_{\cw^{-\al,p_1}} \| g\|_{\cw^{\be,p_2}} \ .$$

\end{lemma}

\subsection{About the resolution of the auxiliary deterministic equation}\label{subsec:solving-irreg}

\

\smallskip

For all $T\geq 0$, $\al,\be>0$ and $p\geq 2$, introduce the space 
\begin{equation}\label{scale-spaces-x}
X^{\alpha,\be, p}(T):=\mathcal{C}([0,T]; \mathcal{W}^{-\alpha,p}(\mathbb{R}^d))\cap \mathcal{C}((0,T]; \mathcal{W}^{\be,p}(\mathbb{R}^d)) \, ,
\end{equation}
equipped with the norm
$$\|v\|_{X(T)}:=\|v\|_{X^{\al,\be,p}(T)}=\|v\|_{L^{\infty}_T \mathcal{W}^{-\al,p}}+\|v\|_{Y(T)}$$
where the $Y(T)$-seminorm is given by
$$\|v\|_{Y(T)}=\underset{0<t\leq T}{\sup} t^{\frac{\be+\al}{2}}\|v(t)\|_{\mathcal{W}^{\be,p}(\mathbb{R}^d)}.$$
Also, remember that the subspace $\mathcal{R}_{\al,p}$ has been defined in~\eqref{defi-space-e-al}.

\smallskip

To end with, we are in a position to formulate (and prove) the main result of this section:

\begin{theorem}\label{thm:noregular}
Suppose that $d\geq 1$ and $p\geq2$ is such that $\frac{d}{2p}<1.$
In addition, assume that $\al, \be >0$ verify
\begin{equation}\label{condi-al}
\al<\be<\min\bigg(2-\al-\frac{d}{p}, 2-2\al\bigg) .
\end{equation}
Then for every $\phi \in \mathcal{W}^{-\alpha,p}(\mathbb{R}^d)$ and $(\luxor,\cherry) \in \mathcal{R}_{\alpha,p}$, one can find a time $T>0$ for which equation~\eqref{ccl} admits a unique solution in the above-defined set $X^{\alpha,\be,p}(T)$.
\end{theorem}

In the same way as in Section~\ref{subsec:determ-regu}, the proof of Theorem~\ref{thm:noregular} is of course a direct consequence of the estimates below for the map $\Gamma_{T, \luxor,\cherry}$ defined for all $T\geq 0$ and $(\luxor,\cherry)\in \mathcal{R}_{\al,p}$ by
\begin{multline}
\Gamma_{T, \luxor,\cherry}(v):=e^{t\Delta}\phi+\int_0^t e^{(t-\tau)\Delta}(\rho^2 v_{\tau}^2)d\tau+2\int_0^t e^{(t-\tau)\Delta}(\rho v_{\tau}\cdot\luxor_{\tau})\, d\tau\\
+\int_0^t e^{(t-\tau)\Delta}(\cherry_{\tau})\, d\tau\, , \quad t\in [0,T] \, .
\end{multline}

\begin{proposition}\label{last}
Suppose that $d\geq 1$ and $p\geq2$ is such that $\frac{d}{2p}<1.$ In addition, assume that $\al, \be >0$ verify condition \eqref{condi-al}. Then there exists $\varepsilon >0$ such that, setting $X(T):=X^{\alpha,\be,p}(T)$, the following bounds hold true: for all $0\leq T\leq 1$, $\phi \in \mathcal{W}^{-\alpha,p}(\mathbb{R}^d), (\luxor_1,\cherry_1) \in \mathcal{R}_{\alpha,p},  (\luxor_2,\cherry_2) \in \mathcal{R}_{\alpha,p}$ and $v, v_1, v_2 \in X(T)$,

\smallskip

\begin{equation}
\|\Gamma_{T, \luxor_1,\cherry_1}(v)\|_{X(T)}\lesssim  \|\phi\|_{\mathcal{W}^{-\alpha,p}}+T^{\varepsilon}\Big[\|v\|_{X(T)}^2+\|\luxor_1\|_{L^\infty_T \mathcal{W}^{-\alpha,p}} \|v\|_{X(T)}+\|\cherry_1\|_{L^\infty_T \mathcal{W}^{-2\alpha,p}}\Big]\, , \label{bound7}
\end{equation}
and
\begin{align}
&\|\Gamma_{T, \luxor_1, \cherry_1}(v_1)-\Gamma_{T, \luxor_2, \cherry_2}(v_2)\|_{X(T)}
\nonumber \\
&\lesssim \ T^{\varepsilon}\Big[\|v_1-v_2\|_{X(T)}\{\|v_1\|_{X(T)}+\|v_2\|_{X(T)}\}+\|\luxor_1-\luxor_2\|_{L^\infty_T \mathcal{W}^{-\alpha,p}}\|v_1\|_{X(T)}\nonumber\\
&\hspace{4cm}+\|\luxor_2\|_{L^\infty_T \mathcal{W}^{-\alpha,p}}\| v_1-v_2\|_{X(T)}+\|\cherry_1-\cherry_2\|_{L^\infty_T \mathcal{W}^{-2\alpha,p}}\Big]\label{bound8} \, , 
\end{align}
where the proportional constants depend only on $\rho$, $p$, $\al$ and $\be$.
\end{proposition}

\begin{proof}[Proof of Proposition~\ref{last}]
Let us bound each of the four terms in the expression of $\Gamma_{T, \luxor,\cherry}$ separately.

\noindent
\textbf{Bound on $e^{t\Delta}\phi$:}
For all $0\leq t \leq T$, using the expression of the heat semigroup, it holds that $\|e^{t\Delta}\phi\|_{\mathcal{W}^{-\al,p}(\mathbb{R}^d)}\lesssim \|\phi\|_{\mathcal{W}^{-\al,p}(\mathbb{R}^d)},$ yielding to
$$\|e^{t\Delta}\phi\|_{L^{\infty}_T \mathcal{W}^{-\al,p}}\lesssim \|\phi\|_{\mathcal{W}^{-\al,p}(\mathbb{R}^d)}.$$
Now, thanks to Proposition \ref{heat}, for any $0<t\leq T$, $\|e^{t\Delta}\phi\|_{\mathcal{W}^{\be,p}(\mathbb{R}^d)}\lesssim t^{-\frac{\be+\al}{2}} \|\phi\|_{\mathcal{W}^{-\al,p}(\mathbb{R}^d)},$ which implies $t^{\frac{\be+\al}{2}}\|e^{t\Delta}\phi\|_{\mathcal{W}^{\be,p}(\mathbb{R}^d)}\lesssim \|\phi\|_{\mathcal{W}^{-\al,p}(\mathbb{R}^d)}$ and
$\|e^{t\Delta}\phi\|_{Y(T)}\lesssim \|\phi\|_{\mathcal{W}^{-\al,p}(\mathbb{R}^d)}$
leading to
\begin{equation}\label{esti1}
\|e^{t\Delta}\phi\|_{X(T)}\lesssim \|\phi\|_{\mathcal{W}^{-\al,p}(\mathbb{R}^d)}.
\end{equation}

\noindent
\textbf{Bound on $\int_0^t e^{(t-\tau)\Delta}(\rho^2 v_{\tau}^2)d\tau$:}
Let $0\leq t \leq T.$ Using successively Minkowsky inequality, Proposition \ref{estimate} and Lemma \ref{lem:product} (since $\al<\be$), we get

\begin{eqnarray*}
\Big\|\int_0^t e^{(t-\tau)\Delta}(\rho^2 v_{\tau}^2)d\tau\Big\|_{\mathcal{W}^{-\al,p}}&\leq&
 \int_0^t \|e^{(t-\tau)\Delta}(\rho^2 v_{\tau}^2)\|_{\mathcal{W}^{-\al,p}}d\tau\\
&\lesssim&\int_0^t\frac{1}{(t-\tau)^{\frac{d}{2p}}}\|\rho^2 v_{\tau}^2\|_{\mathcal{W}^{-\al,\frac{p}{2}}}d\tau \\
&\lesssim&\int_0^t\frac{1}{(t-\tau)^{\frac{d}{2p}}}\|v_{\tau}\|_{\mathcal{W}^{\be,p}}\|v_{\tau}\|_{\mathcal{W}^{-\al,p}}d\tau\\
&\lesssim& \int_0^t (t-\tau)^{-\frac{d}{2p}}\tau^{-\frac{\be+\al}{2}}d\tau\|v\|_{Y(T)}\|v\|_{L^{\infty}_T \mathcal{W}^{-\al,p}}\\
&\lesssim& T^{1-\frac{d}{2p}-\frac{\be+\al}{2}}\|v\|_{X(T)}^2.
\end{eqnarray*}
Taking the supremum over $t \in [0,T],$ we deduce:
$$\Big\|\int_0^t e^{(t-\tau)\Delta}(\rho^2 v_{\tau}^2)d\tau\Big\|_{L^{\infty}_T \mathcal{W}^{-\al,p}}\lesssim  T^{1-\frac{d}{2p}-\frac{\be+\al}{2}}\|v\|_{X(T)}^2.$$
Now, as $e^{(t-\tau)\Delta}=e^{\frac{t-\tau}{2}\Delta}e^{\frac{t-\tau}{2}\Delta}$, the heat kernel allows a gain of both regularity (Proposition \ref{heat}) and integrability (Proposition \ref{estimate}) showing that for all $0<t\leq T$,

\begin{eqnarray*}
\Big\|\int_0^t e^{(t-\tau)\Delta}(\rho^2 v_{\tau}^2)d\tau\Big\|_{\mathcal{W}^{\be,p}}&\leq&
 \int_0^t \|e^{(t-\tau)\Delta}(\rho^2 v_{\tau}^2)\|_{\mathcal{W}^{\be,p}}d\tau\\
&\lesssim& \int_0^t (t-\tau)^{-\frac{\be+\al}{2}}\|e^{\frac{t-\tau}{2}\Delta}(\rho^2 v_{\tau}^2)\|_{\mathcal{W}^{-\al,p}}d\tau\\
&\lesssim&\int_0^t (t-\tau)^{-\frac{\be+\al}{2}-\frac{d}{2p}}\|\rho^2 v_{\tau}^2\|_{\mathcal{W}^{-\al,\frac{p}{2}}}d\tau \\
&\lesssim&\int_0^t(t-\tau)^{-\frac{\be+\al}{2}-\frac{d}{2p}}\|v_{\tau}\|_{\mathcal{W}^{\be,p}}\|v_{\tau}\|_{\mathcal{W}^{-\al,p}}d\tau\\
&\lesssim& \int_0^t (t-\tau)^{-\frac{\be+\al}{2}-\frac{d}{2p}}\tau^{-\frac{\be+\al}{2}}d\tau\|v\|_{Y(T)}\|v\|_{L^{\infty}_T \mathcal{W}^{-\al,p}}\\
&\lesssim& t^{1-\frac{d}{2p}-\be-\al}\|v\|_{X(T)}^2,
\end{eqnarray*}
leading to
\begin{eqnarray*}
t^{\frac{\be+\al}{2}}\Big\|\int_0^t e^{(t-\tau)\Delta}(\rho^2 v_{\tau}^2)d\tau\Big\|_{\mathcal{W}^{\be,p}}&\lesssim& t^{1-\frac{d}{2p}-\frac{\be+\al}{2}}\|v\|_{X(T)}^2\\
&\lesssim&T^{1-\frac{d}{2p}-\frac{\be+\al}{2}}\|v\|_{X(T)}^2.
\end{eqnarray*}
Taking the supremum over $t \in (0,T]$, it holds 
$$\Big\|\int_0^t e^{(t-\tau)\Delta}(\rho^2 v_{\tau}^2)d\tau\Big\|_{Y(T)}\lesssim T^{1-\frac{d}{2p}-\frac{\be+\al}{2}}\|v\|_{X(T)}^2,$$
which entails
\begin{equation}\label{esti2}
\Big\|\int_0^t e^{(t-\tau)\Delta}(\rho^2 v_{\tau}^2)d\tau\Big\|_{X(T)}\lesssim T^{1-\frac{d}{2p}-\frac{\be+\al}{2}}\|v\|_{X(T)}^2.
\end{equation}

\noindent
\textbf{Bound on $\int_0^t e^{(t-\tau)\Delta}(\rho v_{\tau}\cdot\luxor_{\tau})\, d\tau$:}
By repeating the same arguments as above (one of two $\rho v$ being replaced by $\luxor$), we obtain
\begin{equation}\label{esti3}
\Big\|\int_0^t e^{(t-\tau)\Delta}(\rho v_{\tau}\cdot\luxor_{\tau})d\tau\Big\|_{X(T)}\lesssim T^{1-\frac{d}{2p}-\frac{\be+\al}{2}}\|v\|_{X(T)}\|\luxor\|_{L^{\infty}_T \mathcal{W}^{-\al,p}}.
\end{equation}

\noindent
\textbf{Bound on $\int_0^t e^{(t-\tau)\Delta}(\cherry_{\tau}) d\tau$:}
Fix $0\leq t \leq T$. Using successively Minkowsky inequality and Proposition \ref{heat}, we write
\begin{eqnarray*}
\Big\|\int_0^t e^{(t-\tau)\Delta}(\cherry_{\tau})d\tau\Big\|_{\mathcal{W}^{-\al,p}}&
\leq& \int_0^t \|e^{(t-\tau)\Delta}(\cherry_{\tau})\|_{\mathcal{W}^{-\al,p}}d\tau \\
&\lesssim& \int_0^t (t-\tau)^{-\frac{\al}{2}}\|\cherry_{\tau}\|_{\mathcal{W}^{-2\al,p}}d\tau \\
&\lesssim & T^{1-\frac{\al}{2}}\|\cherry\|_{L^{\infty}_T \mathcal{W}^{-2\al,p}},
\end{eqnarray*}
which entails
\begin{equation}\label{fin1}
\Big\|\int_0^t e^{(t-\tau)\Delta}(\cherry_{\tau})d\tau\Big\|_{L^{\infty}_T \mathcal{W}^{-\al,p}}\lesssim T^{1-\frac{\al}{2}}\|\cherry\|_{L^{\infty}_T \mathcal{W}^{-2\al,p}}.
\end{equation}
Then, for all $0<t\leq T$, thanks to the regularising effect of the heat semigroup and the assumption $\be+2\al<2$,
\begin{eqnarray*}
\Big\|\int_0^t e^{(t-\tau)\Delta}(\cherry_{\tau})d\tau\Big\|_{\mathcal{W}^{\be,p}}&
\leq& \int_0^t \|e^{(t-\tau)\Delta}(\cherry_{\tau})\|_{\mathcal{W}^{\be,p}}d\tau \\
&\lesssim& \int_0^t (t-\tau)^{-\frac{\be+2\al}{2}}\|\cherry_{\tau}\|_{\mathcal{W}^{-2\al,p}}d\tau \\
&\lesssim & t^{1-\frac{\be+2\al}{2}}\|\cherry\|_{L^{\infty}_T \mathcal{W}^{-2\al,p}},
\end{eqnarray*}
which leads to
$t^{\frac{\be+\al}{2}} \Big\|\displaystyle\int_0^t e^{(t-\tau)\Delta}(\cherry_{\tau})d\tau\Big\|_{\mathcal{W}^{\be,p}}\lesssim T^{1-\frac{\al}{2}} \|\cherry\|_{L^{\infty}_T \mathcal{W}^{-2\al,p}}$. Thus, 
\begin{equation}\label{fin2}
\Big\|\int_0^t e^{(t-\tau)\Delta}(\cherry_{\tau})d\tau\Big\|_{Y(T)}\lesssim T^{1-\frac{\al}{2}}\|\cherry\|_{L^{\infty}_T \mathcal{W}^{-2\al,p}}
\end{equation}
and, combining the identities (\ref{fin1}) and (\ref{fin2}),
\begin{equation}\label{esti4}
\Big\|\int_0^t e^{(t-\tau)\Delta}(\cherry_{\tau})d\tau\Big\|_{X(T)}\lesssim T^{1-\frac{\al}{2}}  \|\cherry\|_{L^{\infty}_T \mathcal{W}^{-2\al,p}}.
\end{equation}

\noindent
Putting together the four inequalities (\ref{esti1}), (\ref{esti2}), (\ref{esti3}) and (\ref{esti4}) provides us with the desired bound (\ref{bound7}). The second one (\ref{bound8}) can be obtained with similar arguments.

\end{proof}

\section{On the construction of the relevant stochastic objects}\label{stochastic}

The objective of this section is to prove Proposition \ref{sto}, Proposition \ref{sto1} and to give an asymptotic estimate of the quantity $\mathbb{E}\big[\<Psi>_n(t,x)^2\big] $ (Proposition \ref{prop:renorm-cstt}).
We recall that the space-time fractional white noise $B$ has been defined as a Gaussian process whose mean and covariance function are well-known (see Definition \ref{white}). This definition can seem a bit abstract and, in order to realize computations with this noise, we need a representation formula. We will resort to the harmonizable representation of the fractional Brownian motion presented in \cite{harmo}. Let us recall the procedure.
First of all, we fix, on some complete probability space $(\Omega,\mathcal{F},\mathbb{P})$, a space-time white noise $W$ on $\R^{d+1}$. Then let $H=(H_0,H_1,\dots,H_d)\in (0,1)^{d+1}$ and set, for any $t\in [0,T]$ and $x\in \R^d$,
\begin{equation}\label{harmo-repres}
B(t,x_1,\dots,x_d):= c_H\int_{\xi \in \mathbb{R}}\int_{\eta \in \mathbb{R}^d} \frac{e^{\imath t\xi}-1}{|\xi|^{H_0+\frac{1}{2}}}\prod_{i=1}^{d}\frac{e^{\imath x_i\eta_i}-1}{|\eta_i|^{H_i+\frac{1}{2}}}\, \widehat{W}(d\xi,d\eta)\, ,
\end{equation} 
where $c_H>0$ is a constant, and where $\widehat{W}$ stands for the Fourier transform of $W$. 
Then, for an appropriate value of $c_H$, $B$ is a space-time fractional Brownian motion of index $H$ (in the sense of Definition~\ref{white}). 

\smallskip

Unfortunately, we know that $B$ suffers from a lack of regularity measured by its Hurst index. That is why we are looking for a smooth approximation $(B_n)_{n\geq0}$. Our strategy is based on a truncation of the integration domain in~\eqref{harmo-repres}: precisely, we set $B_0\equiv 0$, and for $n\geq 1$,
\begin{equation}\label{defi:approx-b-n}
B_n(t,x_1,\dots,x_d)(\omega):= c_H\int_{|\xi|\leq 2^{2n}}^{}\int_{|\eta|\leq 2^n} \frac{e^{\imath t\xi}-1}{|\xi|^{H_0+\frac{1}{2}}}\prod_{i=1}^{d}\frac{e^{\imath x_i\eta_i}-1}{|\eta_i|^{H_i+\frac{1}{2}}}\, \widehat{W}(d\xi,d\eta)\, .
\end{equation}
It is then clear than $B_n$ defines a smooth process for any $n\geq 0$ and that for all $(t,x)\in [0,T]\times \R^d$, $B_n(t,x) \stackrel{n\to\infty}{\longrightarrow} B(t,x)$ in $L^2(\Omega)$ .

\subsection{Linear solution associated with the model}

\

\smallskip
With a view to constructing the first order process $\<Psi>$, the solution of the linear equation
\begin{equation}\label{linear1}
 \left\{
    \begin{array}{ll}
  \partial_t\<Psi>-\Delta \<Psi> = \dot B, \hskip 0.3 cm t \in [0,T], \hspace{0,2cm} x \in \mathbb{R}^d,   \\
 \  \<Psi>(0,.)=0,
    \end{array}
\right.
\end{equation}
we are now interested in giving a rigorous meaning to the sequence $(\<Psi>_n)_{n\geq 0}$ of the regularised version of (\ref{linear1}), that is the sequence of solutions to 
\begin{equation}\label{regu}
 \left\{
    \begin{array}{ll}
  \partial_t\<Psi>_n-\Delta \<Psi>_n = \dot B_n\, , \hskip 0.3 cm t \in [0,T]\, , \hspace{0,2cm} x \in \mathbb{R}^d\, ,   \\
 \  \<Psi>_n(0,.)=0\, ,
    \end{array}
\right.    
\end{equation}
where, for all $n\geq 0$, $\dot{B}_n$ denotes the classical derivative $\dot{B}_n:=\partial_t\partial_{x_1}\cdots \partial_{x_d}B_n$. Despite the regularity of $\dot B_n$, its {\it a priori} lack of integrability over the whole space $\mathbb{R}^d$ prevents us from using some general theorem which would guarantee the existence of $\<Psi>_n$. We thus propose a stochastic construction of $\<Psi>_n$. Should this latter exist, its expression would by given the formula
$$\<Psi>_n(t,x)=\int_0^t e^{(t-s)\Delta}(\dot B_n(s))(x)ds.$$

A formal computation leads us to the following expression:
\small
\begin{align*}
&\<Psi>_n(t,x)=\int_0^t e^{(t-s)\Delta}(\dot B_n(s))(x)ds\\
&=\int_0^t \mathcal{F}^{-1}(e^{-|\xi|^2 (t-s)}\widehat{\dot B_n}(s,\xi))(x)ds\\
&=\int_0^t \mathcal{F}^{-1}(e^{-|\xi|^2 (t-s)})\star \dot B_n(s,x)ds\\
&=\int_0^t \int_{\mathbb{R}^d}\mathcal{F}^{-1}(e^{-|\xi|^2 (t-s)})(x-y)\dot B_n(s,y)dyds\\
&= c_H\int_{0}^{t}ds\int_{\mathbb{R}^d}dy\, \int_{|\xi|\leq 2^{2n}}\int_{|\eta|\leq 2^n}\mathcal{F}^{-1}(e^{-|\xi|^2 (t-s)})(x-y)i^{d+1}\frac{\xi}{|\xi|^{H_0+\frac{1}{2}}}\prod_{i=1}^{d}\frac{\eta_i}{|\eta_i|^{H_i+\frac{1}{2}}}e^{i\xi s}e^{i\langle \eta,y \rangle}\widehat{W}(d\xi,d\eta)\\
&= c_H i^{d+1}\int_{|\xi|\leq 2^{2n}}^{}\int_{|\eta|\leq 2^n}\frac{\xi}{|\xi|^{H_0+\frac{1}{2}}}\prod_{i=1}^{d}\frac{\eta_i}{|\eta_i|^{H_i+\frac{1}{2}}}e^{i\langle \eta,x \rangle}\cdot\\
&\hspace{2,7cm}\left[\int_{0}^{t}ds \, e^{i\xi s}\left(\int_{\mathbb{R}^d}dy\, \mathcal{F}^{-1}(e^{-|\xi|^2 (t-s)})(x-y)e^{-i\langle \eta,x-y \rangle}\right)\right]\widehat{W}(d\xi,d\eta)\\
&= c_H i^{d+1}\int_{|\xi|\leq 2^{2n}}^{}\int_{|\eta|\leq 2^n}\frac{\xi}{|\xi|^{H_0+\frac{1}{2}}}\prod_{i=1}^{d}\frac{\eta_i}{|\eta_i|^{H_i+\frac{1}{2}}}e^{i\langle \eta,x \rangle}\cdot\left[\int_{0}^{t}ds \, e^{i\xi(t-s)}\left(\int_{\mathbb{R}^d}dy\, \mathcal{F}^{-1}(e^{-|\xi|^2 s})(y)e^{-i\langle \eta,y \rangle}\right)\right]\widehat{W}(d\xi,d\eta).
\end{align*}
We have finally obtained that
$$\<Psi>_n(t,x)=c_H i^{d+1}\int_{|\xi|\leq 2^{2n}}^{}\int_{|\eta|\leq 2^n}\frac{\xi}{|\xi|^{H_0+\frac{1}{2}}}\prod_{i=1}^{d}\frac{\eta_i}{|\eta_i|^{H_i+\frac{1}{2}}}e^{i\langle \eta,x \rangle}\gamma_t(\xi,|\eta|)\, \widehat{W}(d\xi,d\eta),$$
where for every $t\geq 0$, $\xi \in \mathbb{R}$ and $r>0$, we introduce the quantity $\gamma_t(\xi,r)$ as 
\begin{equation}
\gamma_t(\xi,r):=e^{i\xi t}\int_{0}^{t}e^{-s r^2}e^{-i \xi s }ds \, .
\end{equation}
\normalsize
Moreover, let us observe that $\<Psi>$ is real-valued. Indeed, let $f_{t,x}$ the function defined for all $\xi \in \mathbb{R}, \eta \in \mathbb{R}^d$ by
$$f_{t,x}(\xi,\eta):=c_H i^{d+1}\mathbbm{1}_{|\xi|\leq 2^{2n}}^{}\mathbbm{1}_{|\eta|\leq 2^n}\frac{\xi}{|\xi|^{H_0+\frac{1}{2}}}\prod_{i=1}^{d}\frac{\eta_i}{|\eta_i|^{H_i+\frac{1}{2}}}e^{i\langle \eta,x \rangle}\gamma_t(\xi,|\eta|),$$
in such a way that $\displaystyle\<Psi>_n(t,x)=\int_{\mathbb{R}^{d+1}}f_{t,x}(\xi,\eta)\widehat{W}(d\xi,d\eta).$ As $\overline{f_{t,x}(\xi,\eta)}=f_{t,x}(-\xi,-\eta),$ the Fourier transform of $f$ is real-valued that implies
\begin{eqnarray*}
\overline{\<Psi>_n(t,x)}&=&\overline{\int_{\mathbb{R}^{d+1}}f_{t,x}(\xi,\eta)\widehat{W}(d\xi,d\eta)}\\
&=&\overline{\int_{\mathbb{R}^{d+1}}\widehat{f_{t,x}}(\xi,\eta)W(d\xi,d\eta)}\\
&=&\int_{\mathbb{R}^{d+1}}\overline{\widehat{f_{t,x}}(\xi,\eta)}W(d\xi,d\eta)\\
&=&\int_{\mathbb{R}^{d+1}}\widehat{f_{t,x}}(\xi,\eta)W(d\xi,d\eta)\\
&=&\int_{\mathbb{R}^{d+1}}f_{t,x}(\xi,\eta)\widehat{W}(d\xi,d\eta)\\
&=&\<Psi>_n(t,x).
\end{eqnarray*}
We are now in a position to define properly $\<Psi>_n$.
\begin{definition}\label{defi:luxo-n}
We call a solution of equation \eqref{regu} (or \emph{linear solution associated with \eqref{dim-d}}) any centered real Gaussian process 
$$\big\{\<Psi>_n(s,x), \, n\geq 1, \, s\geq 0, \, x\in \R^d \big\}$$
whose covariance function is given by the relation: for all $n,m\geq 1$, $s, t\geq 0$ and $x, y \in \mathbb{R}^d,$
\begin{equation}\label{cova-luxo}
\mathbb{E}\Big[\<Psi>_n(s,x)\<Psi>_m(t,y)\Big]=c_{H}^2\int_{(\xi,\eta)\in D_n \cap D_m}\frac{1}{|\xi|^{2H_0-1}}\prod\limits_{i=1}^{d}\frac{1}{|\eta_i|^{2H_i-1}}\gamma_{s}(\xi,|\eta|)\overline{\gamma_{t}(\xi,|\eta|)}e^{\imath \langle \eta, x-y \rangle}\, d\xi d\eta \, ,
\end{equation}
where $D_n:=B_{2n}^1 \times B_n^d$ with $B_\ell^k:=\left\{ \lambda \in \mathbb{R}^k: |\lambda| \leq 2^\ell \right\}$.
\end{definition}

\subsection{Some preliminary estimates on $\ga$}

\

\smallskip
Before proving that $\chi \<Psi>_n$ is a Cauchy sequence in a convenient subspace, we need to establish some bounds on the quantity $\gamma_t(\xi,r)$ which contains all the information with respect to the heat semigroup:

\begin{equation}
\gamma_t(\xi,r):=e^{i\xi t}\int_{0}^{t}e^{-s r^2}e^{-i \xi s }ds \, .
\end{equation}
We first state some estimates on the variation
\begin{equation}
\gamma_{s,t}(\xi,r):=\gamma_t(\xi,r)-\gamma_s(\xi,r).
\end{equation}

\begin{lemma}\label{lem}
For all $0\leq s \leq t\leq 1$, $\xi \in \mathbb{R}$, $r>0$ and $\varepsilon \in [0,1]$, the following bound holds true:
$$
|\gamma_{s,t}(\xi,r)|\lesssim \min\bigg( |\xi|^{\varepsilon}|t-s|^{\varepsilon} + |t-s|,\frac{|t-s| r^2 }{|\xi|}+  \frac{|t-s|^{\varepsilon}\{1+r^2\}}{|\xi|^{1-\varepsilon}},\frac{|t-s|^{\varepsilon}\{r^{2\varepsilon}+|\xi|^{\varepsilon}\}}{\sqrt{r^4+\xi^2}}\bigg)\, .
$$
\end{lemma}

\begin{proof} First of all, we write
$$\gamma_{s,t}(\xi,r)=\{e^{i\xi t}-e^{i\xi s}\}\int_{0}^{s}e^{-u r^2}e^{-i \xi u }du + e^{i \xi t}\int_{s}^{t}e^{-u r^2}e^{-i \xi u }du \, ,$$ 
and so
$$
|\gamma_{s,t}(\xi,r)|\lesssim |e^{i\xi t}-e^{i\xi s}|\bigg|\int_{0}^{s}e^{-u r^2}e^{-i \xi u }du\bigg| +\bigg|\int_{s}^{t}e^{-u r^2}e^{-i \xi u }du\bigg|\lesssim |\xi|^{\varepsilon}|t-s|^{\varepsilon} + |t-s|\, .
$$
Then observe that 
$$\gamma_{t}(\xi,r)=e^{i\xi t}\int_{0}^{t}e^{-s r^2}e^{-i \xi s }ds=-\frac{e^{- r^2 t}-e^{i \xi t}}{i\xi}+\frac{i e^{i \xi t}r^2}{\xi}\int_{0}^{t}e^{-s r^2}e^{-i \xi s }ds \, , $$
which readily entails 
\begin{align*}
&\gamma_{s,t}(\xi,r)\\
&=-\frac{\{e^{- r^2 t}-e^{- r^2 s}\}-\{e^{i \xi t}-e^{i \xi s}\}}{i\xi}+\frac{i r^2}{\xi}\{e^{i\xi t}-e^{i\xi s}\}\int_{0}^{s}e^{-u r^2}e^{-i \xi u }du+\frac{i e^{i \xi t}r^2}{\xi}\int_{s}^{t}e^{-u r^2}e^{-i \xi u }du \, .
\end{align*}
Thus,
\begin{eqnarray*}
|\gamma_{s,t}(\xi,r)|&\lesssim& r^2 \frac{|t-s|}{|\xi|}+\frac{|t-s|^{\varepsilon}}{|\xi|^{1-\varepsilon}}+r^2\frac{|t-s|^{\varepsilon}}{|\xi|^{1-\varepsilon}}+\frac{r^2}{|\xi|}|t-s|\\
&\lesssim& r^2 \frac{|t-s|}{|\xi|}+ \{1+r^2\} \frac{|t-s|^{\varepsilon}}{|\xi|^{1-\varepsilon}}\, .
\end{eqnarray*}
To end with, it can be verified that 
$$\gamma_{t}(\xi,r)=\frac{e^{i \xi t}-e^{- r^2 t}}{r^2+i \xi} ,$$
which yields
\begin{eqnarray*}
|\gamma_{s,t}(\xi,r)|&=& \frac{1}{\sqrt{r^4+\xi^2}}\Big|\{e^{- r^2 t}-e^{- r^2 s}\}-\{e^{i \xi t}-e^{i \xi s}\}\Big|\\
&\lesssim&\frac{|t-s|^{\varepsilon}}{\sqrt{r^4+\xi^2}}\{r^{2\varepsilon}+|\xi|^{\varepsilon}\}\  \, .
\end{eqnarray*}
\end{proof}

\begin{corollary}\label{tec}
For any $0\leq s \leq t \leq 1$, $H\in (0,1)$, $r>0$ and $\varepsilon\in [0, H)$, it holds that
$$ \int_{\mathbb{R}}\frac{|\gamma_{s,t}(\xi,r)|^2}{|\xi|^{2H-1}}\, d\xi \lesssim \frac{|t-s|^{2\varepsilon}}{1+r^{4(H-\varepsilon)}}\, .$$
\end{corollary}

\begin{proof}
The desired inequality comes of course from a relevant use of the estimates shown in Lemma \ref{lem}.

\smallskip

For $0<r<1$, we have
$$\int_{\mathbb{R}}\frac{|\gamma_{s,t}(\xi,r)|^2}{|\xi|^{2H-1}}\, d\xi\lesssim |t-s|^{2\varepsilon}\bigg[\int_{|\xi|\leq 1}\frac{d\xi}{|\xi|^{2H-1}}+\int_{|\xi|\geq 1}\frac{d\xi}{|\xi|^{2(H-\varepsilon)+1}}\bigg]\lesssim |t-s|^{2\varepsilon}\, .$$
Then, for $r>1$, it holds that
\begin{eqnarray*}
\int_{\mathbb{R}}\frac{|\gamma_{s,t}(\xi,r)|^2}{|\xi|^{2H-1}}\, d\xi &\lesssim& |t-s|^{2\varepsilon}\int_{\mathbb{R}} \frac{r^{4\varepsilon}+|\xi|^{2\varepsilon}}{(r^4+\xi^2)|\xi|^{2H-1}}\, d\xi \\
&\lesssim& \frac{|t-s|^{2\varepsilon}}{r^{4(H-\varepsilon)}}\int_{\mathbb{R}} \frac{1+|\xi|^{2\varepsilon}}{(1+\xi^2)|\xi|^{2H-1}}\, d\xi \ \lesssim \ \frac{|t-s|^{2\varepsilon}}{r^{4(H-\varepsilon)}}\, .
\end{eqnarray*}
\end{proof}

\subsection{Construction of the first order stochastic process}\label{subsec:proof-sto}
\begin{proof}[Proof of Proposition~\ref{sto}]
Without loss of generality, we will suppose during the whole proof that $T= 1$ and we set, for all $m, n\geq 0$, $\<Psi>_{n,m}:=\<Psi>_m-\<Psi>_n$. Also, along the statement of the proposition, we fix $\al$ verifying~\eqref{assump-gene-al}.

\

\noindent
\textbf{Step 1:}
The first objective is to show that for any $p\geq 1$, $1\leq n \leq m$, $0\leq s \leq t \leq 1$ and $\varepsilon>0$ small enough, it holds
\begin{equation}\label{bou-psi-step-1}
 \int_{\mathbb{R}^d}\mathbb{E}\bigg[\Big|\mathcal{F}^{-1}\Big(\{1+|.|^2\}^{-\frac{\alpha}{2}}\mathcal{F}\big(\chi \big[\<Psi>_{n,m}(t,.)-\<Psi>_{n,m}(s,.)\big]\big)\Big)(x)\Big|^{2p}\bigg]\, dx \lesssim 2^{-4n \varepsilon p}|t-s|^{2 \varepsilon p}\, ,
\end{equation}
where the proportional constant only depends on $p, \alpha$ and $\chi$.

\

Resorting to the same kind of arguments as those used in \cite[Proof of Prop 1.2]{deya3} , we obtain the estimate

\begin{align*}
& \int_{\mathbb{R}^d}dx\, \mathbb{E}\bigg[\Big|\mathcal{F}^{-1}\Big(\{1+|.|^2\}^{-\frac{\alpha}{2}}\mathcal{F}\big(\chi \big[\<Psi>_{n,m}(t,.)-\<Psi>_{n,m}(s,.)\big]\big)\Big)(x)\Big|^{2p}\bigg]\\
&\lesssim  \left(\int_{(\xi,\eta)\in D_{n,m}}\frac{1}{|\xi|^{2H_0-1}}\prod_{i=1}^{d}\frac{1}{|\eta_i|^{2H_i-1}}\{1+|\eta|^2\}^{-\alpha}|\gamma_{s,t}(\xi,|\eta|)|^2\, d\xi d\eta\right)^p \, .
\end{align*}

\smallskip

Now we can split the latter integral into
\begin{align}
& \left(\int_{(\xi,\eta)\in D_{n,m}}\frac{1}{|\xi|^{2H_0-1}}\prod_{i=1}^{d}\frac{1}{|\eta_i|^{2H_i-1}}\{1+|\eta|^2\}^{-\alpha}|\gamma_{s,t}(\xi,|\eta|)|^2\, d\xi d\eta\right)^p \nonumber\\
&\lesssim  \left(\int_{2^{2n}\leq |\xi|\leq 2^{2m}}\int_{|\eta|\leq 2^m}\frac{1}{|\xi|^{2H_0-1}}\prod_{i=1}^{d}\frac{1}{|\eta_i|^{2H_i-1}}\{1+|\eta|^2\}^{-\alpha}|\gamma_{s,t}(\xi,|\eta|)|^2 \, d\xi d\eta\right)^p \nonumber\\
&\hspace{1cm}+ \left(\int_{|\xi|\leq 2^{2m}}\int_{2^n\leq |\eta|\leq 2^m}\frac{1}{|\xi|^{2H_0-1}}\prod_{i=1}^{d}\frac{1}{|\eta_i|^{2H_i-1}}\{1+|\eta|^2\}^{-\alpha}|\gamma_{s,t}(\xi,|\eta|)|^2\, d\xi d\eta\right)^p \nonumber\\
&=:(\mathbb{I}_{n,m}(s,t))^p+(\mathbb{II}_{n,m}(s,t))^p \, . \label{decompos-i-ii}
\end{align}
Let us focus on the estimation of $\mathbb{I}_{n,m}(s,t)$. To this end, we fix $\varepsilon>0$, so that
\begin{equation}\label{element-bounding}
\mathbb{I}_{n,m}(s,t)\leq 2^{-4n\varepsilon}  \int_{\mathbb{R}}d\xi\int_{\R^d}d\eta \, \frac{1}{|\xi|^{2H_0-2\varepsilon-1}}\prod_{i=1}^{d}\frac{1}{|\eta_i|^{2H_i-1}}\{1+|\eta|^2\}^{-\alpha}|\gamma_{s,t}(\xi,|\eta|)|^2\, ,
\end{equation}
which, after a classical hyperspherical change of variables, leads us to
\begin{equation}\label{chg-of-var}
\mathbb{I}_{n,m}(s,t)\lesssim 2^{-4n\varepsilon} \int_{0}^{\infty}dr \, \frac{\left\{1+r^2\right\}^{-\alpha}}{r^{2(H_1+\cdots+H_d)-2d+1}}\left(\int_{\mathbb{R}}d\xi\, \frac{|\gamma_{s,t}(\xi,r)|^2}{|\xi|^{2H_0-2\varepsilon-1}}\right) \, .
\end{equation}
We can now apply Corollary~\ref{tec} with $H:=H_0-\varepsilon$, which gives, for all $0<\varepsilon< \frac{H_0}{2}$,
\begin{align}
&\mathbb{I}_{n,m}(s,t)\nonumber\\
&\lesssim 2^{-4n\varepsilon} |t-s|^{2\varepsilon}\int_{0}^{\infty}dr \, \frac{\left\{1+r^2\right\}^{-\alpha}}{r^{2(H_1+\cdots+H_d)-2d+1}}\frac{1}{1+r^{4H_0-8\varepsilon}}\nonumber\\
&\lesssim 2^{-4n\varepsilon}|t-s|^{2\varepsilon}\left(\int_0^1 \frac{1}{r^{2(H_1+\cdots+H_d)-2d+1}}dr + \int_{1}^{\infty}\frac{1}{r^{2\alpha+2(2H_0+H_1+\cdots+H_d)-2d+1-8\varepsilon}}dr \right)\, .\label{boun-i-pr}
\end{align}
Due to our assumption~\eqref{assump-gene-al}, we can in fact pick $\varepsilon$ small enough so that
$$4\varepsilon< \alpha-\bigg[d-\bigg(2H_0+\sum_{i=1}^{d}H_i\bigg)\bigg] \, ,$$
and for such a parameter, the two integrals in~\eqref{boun-i-pr} are finite, yielding
$$\mathbb{I}_{n,m}(s,t)\lesssim 2^{-4n\varepsilon}|t-s|^{2\varepsilon}\, .$$
It is not difficult to see that the previous estimates could also be used to control $\mathbb{II}_{n,m}(s,t)$, yielding the very same estimate
$$\mathbb{II}_{n,m}(s,t)\lesssim 2^{-4n\varepsilon}|t-s|^{2\varepsilon}\, .$$
Coming back to~\eqref{decompos-i-ii}, we get the desired bound~\eqref{bou-psi-step-1}.

\

\noindent
\textbf{Step 2:}
The estimate we proved in the previous step can be reformulated in the following way
\begin{equation}\label{recap-step-1}
\mathbb{E}\Big[\big\|\chi \<Psi>_{n,m}(t,.)-\chi\<Psi>_{n,m}(s,.)\big\|_{\mathcal{W}^{-\alpha,2p}}^{2p}\Big]\lesssim 2^{-4n \varepsilon p}|t-s|^{2 \varepsilon p}\, ,
\end{equation}
for all $p\geq 1$, $1\leq n \leq m$, $0\leq s \leq t \leq 1$ and $\varepsilon>0$ small enough.

\smallskip

Combining Kolmogorov continuity criterion with the classical Garsia-Rodemich-Rumsey estimate (see~\cite{GRR}), we easily show that for any $p\geq 1$ large enough,
\begin{equation}\label{bou-in-l-p}
\|\chi\<Psi>_{n,m}\|_{ L^{2p}(\Omega; \mathcal{C}_T\mathcal{W}^{-\alpha,2p})} \lesssim 2^{-2n \varepsilon }\, .
\end{equation}
In particular, $(\chi\<Psi>_{n})_{n\geq 1}$ is a Cauchy sequence in $L^{2p}(\Omega; \mathcal{C}([0,T]; \mathcal{W}^{-\alpha,2p}(\mathbb{R}^d)))$ (for any $p\geq 1$ large enough). As $L^{2p}(\Omega; \mathcal{C}([0,T]; \mathcal{W}^{-\alpha,2p}(\mathbb{R}^d)))$ is a Banach space, we deduce the convergence of $(\chi\<Psi>_{n})_{n\geq 1}$ in this space to a limit $\chi\<Psi>$. Coming back to~\eqref{bou-in-l-p}, we also have 
$$
\|\chi\<Psi>-\chi\<Psi>_{n}\|_{ L^{2p}(\Omega; \mathcal{C}_T\mathcal{W}^{-\alpha,2p})}  \lesssim 2^{-2n \varepsilon }\, ,
$$
and from there, a classical use of the Borell-Cantelli lemma justifies the desired almost sure convergence of $(\chi\<Psi>_{n})_{n\geq 1}$ to $\chi\<Psi>$ in $\mathcal{C}([0,T]; \mathcal{W}^{-\alpha,p}(\mathbb{R}^d))$, for every $2\leq p<\infty$. To end with, the convergence in $\mathcal{C}([0,T]; \mathcal{W}^{-\alpha,\infty}(\mathbb{R}^d))$ is the result of the Sobolev embedding $\cw^{-\al+\frac{d}{p}+\eta,p}(\R^d) \subset \cw^{-\al,\infty}(\R^d)$, for any $\eta>0$.

\

\end{proof}

\subsection{Construction of the second order stochastic process}
\begin{proof}[Proof of Proposition~\ref{sto1}]
We will follow the same general strategy as in the proof of Proposition \ref{sto}. Let us again suppose that $T= 1$, and set, for every $m, n\geq 0$, $\<Psi2>_{n,m}:=\<Psi2>_m-\<Psi2>_n$.

\

\noindent
\textbf{Step 1:}
Our first objective here is to prove that for every $p\geq 1$, $0\leq n \leq m$, $0\leq s \leq t \leq 1$, $\varepsilon>0$ small enough, and for every $\al$ verifying
\begin{equation}\label{assump-al-bis}
d-\bigg(2H_0+\sum_{i=1}^{d}H_i\bigg)<\al <\frac14 \, ,
\end{equation}
it holds
\begin{equation}\label{mom-estim-wick-sq}
 \int_{\mathbb{R}^d}\mathbb{E}\bigg[\Big|\mathcal{F}^{-1}\Big(\{1+|.|^2\}^{-\alpha}\mathcal{F}\big(\chi^2 \big[\<Psi2>_{n,m}(t,.)-\<Psi2>_{n,m}(s,.)\big]\big)\Big)(x)\Big|^{2p}\bigg]dx \lesssim 2^{-4n \varepsilon p}|t-s|^{\varepsilon p}\, ,
\end{equation}
where the proportional constant only depends on $p, \alpha,$ and $\chi$.

\smallskip

In order to reduce the length of the proof, we will only show estimate \eqref{mom-estim-wick-sq} for $n=0$, that is we will focus on the bound for the time-variation $\<Psi2>_{m}(t,.)-\<Psi2>_{m}(s,.)$, with $m\geq 1$. The extended result to all $m\geq n\geq 0$ could in fact be easily deduced from the estimates below combined with the bounding argument used (for example) in \eqref{element-bounding}.

\

Resorting to the same kind of arguments as those used in \cite[Proof of Prop 1.6]{deya-wave}, we obtain the estimate
\begin{equation}\label{bou-sum-cj-i-p}
\int_{\R^d} dx \, \mathbb{E}\bigg[\Big|\mathcal{F}^{-1}\Big(\{1+|.|^2\}^{-\al}\mathcal{F}\big(\chi^2 \big[\<Psi2>_{m}(t,.)-\<Psi2>_{m}(s,.)\big]\big)\Big)(x)\Big|^{2}\bigg]^p\lesssim \bigg( \sum_{\ell=1}^4 \cj^\ell_{m;s,t}\bigg)^p
\end{equation}
where
\begin{align}
&\cj^\ell_{m;s,t}:=\int_{(\xi,\eta)\in D_{m}}d\xi d\eta \int_{(\xiti,\etati)\in D_{m}}d\tilde{\xi} d\tilde{\eta} \, \frac{1}{|\xi|^{2H_0-1}}\prod_{i=1}^{d}\frac{1}{|\eta_i|^{2H_i-1}}\frac{1}{|\tilde{\xi}|^{2H_0-1}}\prod_{i=1}^{d}\frac{1}{|\tilde{\eta_i}|^{2H_i-1}}\nonumber\\
&\hspace{7cm} \big|\gga^\ell_{m;s,t}(\xi,|\eta|;\xiti,|\etati|)\big|\big\{1+|\eta-\tilde{\eta}|^2\big\}^{-2\alpha}, \,  \label{defi-cj-i}
\end{align}
with $\gga^\ell_{m;s,t}=\gga^\ell_{m;s,t}(\xi,|\eta|;\xiti,|\etati|)$ given by
\small
\begin{align*}
\gga^1_{m;s,t}:=\gamma_{t}(\xi,|\eta|)\, \overline{\gamma_{s,t}(\xi,|\eta|)}\, |\gamma_t(\xiti,|\etati|)|^2\, ,& \quad \gga^2_{m;s,t}:=\gamma_t(\xi,|\eta|)\, \overline{\gamma_{s,t}(\xi,|\eta|)}\, \overline{\gamma_{s}(\xiti,|\etati|)}\, \gamma_t(\xiti,|\etati|)\, ,\\
\gga^3_{m;s,t}:=\gamma_{s}(\xi,|\eta|)\, \overline{\gamma_{s,t}(\xi,|\eta|)}\, \overline{\gamma_t(\xiti,|\etati|)} \, \gamma_s(\xiti,|\etati|)\, , &\quad \gga^4_{m;s,t}:=\overline{\gamma_{t,s}(\xi,|\eta|)}\, \gamma_s(\xi,|\eta|)|\gamma_s(\xiti,|\etati|)|^2\, \,   \, .
\end{align*}
\normalsize

\smallskip

Let us focus on the treatment of $\cj^1_{m;s,t}$.
\begin{align}
\cj^1_{m;s,t}&\lesssim  \int_{(\xi,\eta)\in D_{m}}d\xi d\eta \int_{(\xiti,\etati)\in D_{m}}d\tilde{\xi} d\tilde{\eta} \, \big\{1+|\eta-\tilde{\eta}|^2\big\}^{-2\alpha}  \nonumber\\
&\hspace{0.5cm}\bigg(\frac{1}{|\xi|^{2H_0-1}}\prod_{i=1}^{d}\frac{1}{|\eta_i|^{2H_i-1}}|\gamma_t(\xi,|\eta|)||\gamma_{s,t}(\xi,|\eta|)|\bigg)\cdot \bigg(\frac{1}{|\tilde{\xi}|^{2H_0-1}}\prod_{i=1}^{d}\frac{1}{|\tilde{\eta_i}|^{2H_i-1}}|\gamma_t(\tilde{\xi},|\tilde{\eta}|)|^2\bigg)  \nonumber\\
&\lesssim  \int_{(\mathbb{R}\times\mathbb{R}^d)^2}d\xi d\eta d\tilde{\xi} d\tilde{\eta} \, \big\{1+||\eta|-|\tilde{\eta}||^2\big\}^{-2\alpha}  \nonumber\\
&\hspace{0.5cm}\bigg(\frac{1}{|\xi|^{2H_0-1}}\prod_{i=1}^{d}\frac{1}{|\eta_i|^{2H_i-1}}|\gamma_t(\xi,|\eta|)||\gamma_{s,t}(\xi,|\eta|)|\bigg)\cdot\bigg(\frac{1}{|\tilde{\xi}|^{2H_0-1}}\prod_{i=1}^{d}\frac{1}{|\tilde{\eta_i}|^{2H_i-1}}|\gamma_t(\tilde{\xi},|\tilde{\eta}|)|^2\bigg) \, , \label{boun-cj-1-m}
\end{align}
where we have used the trivial bound $|\eta-\etati|\geq | |\eta|-|\etati| |$.

\

Now, let us decompose the latter integration domain into $(\R\times \mathbb{R}^d )^2:=D_1\cup D_2$, where
$$D_1:=\Big\{(\xi,\eta,\tilde{\xi},\tilde{\eta}):\  0\leq |\tilde{\eta}|\leq \frac{|\eta|}{2} \ \text{or} \ |\tilde{\eta}|\geq \frac{3|\eta|}{2}\Big\}$$ and
$$D_2:=\Big\{(\xi,\eta,\tilde{\xi},\tilde{\eta}):\  \frac{|\eta|}{2}<|\tilde{\eta}|<\frac{3|\eta|}{2}\Big\}\, .$$
Concerning the integral on $D_1$, the inequality $||\eta|-|\tilde{\eta}||\geq\max(\frac{|\eta|}{2},\frac{|\tilde{\eta}|}{3})$ (valid for all $(\xi,\eta,\tilde{\xi},\tilde{\eta}) \in D_1$) allows us to write
\begin{align*}
\ca_1&:=\int_{D_1}\frac{d\xi d\eta d\tilde{\xi} d\tilde{\eta}}{\{1+||\eta|-|\tilde{\eta}||^2\}^{2\alpha}}\bigg(\frac{1}{|\xi|^{2H_0-1}}\prod_{i=1}^{d}\frac{1}{|\eta_i|^{2H_i-1}}|\gamma_t(\xi,|\eta|)||\gamma_{s,t}(\xi,|\eta|)|\bigg)\cdot\\
&\hspace{7cm}\bigg(\frac{1}{|\tilde{\xi}|^{2H_0-1}}\prod_{i=1}^{d}\frac{1}{|\tilde{\eta_i}|^{2H_i-1}}|\gamma_t(\tilde{\xi},|\tilde{\eta}|)|^2\bigg) \nonumber \\
& \lesssim \Bigg(\int_{\mathbb{R}\times\mathbb{R}^d}\frac{d\xi d\eta }{\{1+|\eta|^2\}^{\alpha}}\frac{1}{|\xi|^{2H_0-1}}\prod_{i=1}^{d}\frac{1}{|\eta_i|^{2H_i-1}}|\gamma_t(\xi,|\eta|)||\gamma_{s,t}(\xi,|\eta|)|  \Bigg)\cdot\nonumber \\
&\hspace{5cm} \Bigg(\int_{\mathbb{R}\times\mathbb{R}^d}\frac{d\tilde{\xi} d\tilde{\eta}}{\{1+|\tilde{\eta}|^2\}^{\alpha}}\frac{1}{|\tilde{\xi}|^{2H_0-1}}\prod_{i=1}^{d}\frac{1}{|\tilde{\eta_i}|^{2H_i-1}}|\gamma_t(\tilde{\xi},|\tilde{\eta}|)|
^2\Bigg)\nonumber \\
&\lesssim \Bigg(\int_{\mathbb{R}\times\mathbb{R}^d}\frac{d\xi d\eta }{\{1+|\eta|^2\}^{\alpha}}\frac{1}{|\xi|^{2H_0-1}}\prod_{i=1}^{d}\frac{1}{|\eta_i|^{2H_i-1}}|\gamma_t(\xi,|\eta|)|^2  \Bigg)^{\frac{1}{2}}\cdot\nonumber \\ 
&\hspace{2cm} \Bigg(\int_{\mathbb{R}\times\mathbb{R}^d}\frac{d\xi d\eta }{\{1+|\eta|^2\}^{\alpha}}\frac{1}{|\xi|^{2H_0-1}}\prod_{i=1}^{d}\frac{1}{|\eta_i|^{2H_i-1}}|\gamma_{s,t}(\xi,|\eta|)|^2  \Bigg)^{\frac{1}{2}}\cdot\nonumber \\ 
&\hspace{4cm}\Bigg(\int_{\mathbb{R}\times\mathbb{R}^d}\frac{d\tilde{\xi} d\tilde{\eta}}{\{1+|\tilde{\eta}|^2\}^{\alpha}}\frac{1}{|\tilde{\xi}|^{2H_0-1}}\prod_{i=1}^{d}\frac{1}{|\tilde{\eta_i}|^{2H_i-1}}|\gamma_t(\tilde{\xi},|\tilde{\eta}|)|^2\Bigg)\, ,
\end{align*}
where Cauchy-Schwarz inequality has been used to get the last estimate. 

\smallskip

Now remark that we are here coping with the same integral as in the proof of Proposition \ref{sto} (see in particular \eqref{element-bounding}) and therefore we can reproduce the arguments in \eqref{chg-of-var}-\eqref{boun-i-pr} to obtain the estimate we are looking for, namely: for all $\varepsilon >0$ small enough,
$$\ca_1\lesssim |t-s|^{\varepsilon} \, .$$

\

Concerning the integral over $D_2$, a hyperspherical change of variable entails
\begin{align*}
&\int_{\frac{|\eta|}{2}<|\tilde{\eta}|<\frac{3|\eta|}{2}}\frac{d\tilde{\eta}}{\{1+||\eta|-|\tilde{\eta}||^2\}^{2\alpha}}|\gamma_t(\tilde{\xi},|\tilde{\eta}|)|^2 \prod\limits_{i=1}^{d}\frac{1}{|\tilde{\eta_i}|^{2H_i-1}} \nonumber \\
&=|\eta|^{-2(H_1+...+H_d)+2d}\int_{\frac{1}{2}<|\tilde{\eta}|<\frac{3}{2}}\frac{d\tilde{\eta}}{\{1+|\eta|^2(1-|\tilde{\eta}|)^2\}^{2\alpha}}|\gamma_t(\tilde{\xi},|\eta||\tilde{\eta}|)|^2 \prod\limits_{i=1}^{d}\frac{1}{|\tilde{\eta_i}|^{2H_i-1}} \nonumber \\
&\lesssim |\eta|^{-2(H_1+...+H_d)+2d}\int_{\frac{1}{2}<r<\frac{3}{2}}\frac{dr}{\{1+|\eta|^2(1-r)^2\}^{2\alpha}}|\gamma_t(\tilde{\xi},|\eta|r)|^2\, .
\end{align*}
As a consequence,
\begin{align*}
\ca_2&:=\int_{D_2}\frac{d\xi d\eta d\tilde{\xi} d\tilde{\eta}}{\{1+||\eta|-|\tilde{\eta}||^2\}^{2\alpha}}\bigg(\frac{1}{|\xi|^{2H_0-1}}\prod_{i=1}^{d}\frac{1}{|\eta_i|^{2H_i-1}}|\gamma_t(\xi,|\eta|)||\gamma_{s,t}(\xi,|\eta|)|\bigg)\cdot\\
&\hspace{7cm}\bigg(\frac{1}{|\tilde{\xi}|^{2H_0-1}}\prod_{i=1}^{d}\frac{1}{|\tilde{\eta_i}|^{2H_i-1}}|\gamma_t(\tilde{\xi},|\tilde{\eta}|)|^2\bigg) \nonumber \\
&\lesssim \int_{\mathbb{R}^d}\frac{d\eta}{|\eta|^{2(H_1+...+H_d)-2d}}\prod_{i=1}^{d}\frac{1}{|\eta_i|^{2H_i-1}}\int_{\frac{1}{2}<r<\frac{3}{2}}\frac{dr}{\{1+|\eta|^2(1-r)^2\}^{2\alpha}}\cdot \nonumber \\
&\hspace{5cm}\bigg(\int_{\mathbb{R}}d\xi\, \frac{|\gamma_t(\xi,|\eta|)||\gamma_{s,t}(\xi,|\eta|)|}{|\xi|^{2H_0-1}}\bigg)\cdot \bigg(\int_{\mathbb{R}}d\tilde{\xi}\, \frac{|\gamma_t(\tilde{\xi},|\eta|r)|^2}{|\tilde{\xi}|^{2H_0-1}}\bigg) \nonumber \\
&\lesssim \int_{\mathbb{R}^d}\frac{d\eta}{|\eta|^{2(H_1+...+H_d)-2d}}\prod\limits_{i=1}^{d}\frac{1}{|\eta_i|^{2H_i-1}}\int_{\frac{1}{2}<r<\frac{3}{2}}\frac{dr}{\{1+|\eta|^2(1-r)^2\}^{2\alpha}} \cdot\nonumber \\
&\hspace{2cm}\bigg(\int_{\mathbb{R}}d\xi\,  \frac{|\gamma_t(\xi,|\eta|)|^2}{|\xi|^{2H_0-1}}\bigg)^{\frac{1}{2}}\cdot \bigg(\int_{\mathbb{R}}d\xi\,  \frac{|\gamma_{s,t}(\xi,|\eta|)|^2}{|\xi|^{2H_0-1}}\bigg)^{\frac{1}{2}}\cdot \bigg(\int_{\mathbb{R}}d\tilde{\xi}\, \frac{|\gamma_t(\tilde{\xi},|\eta|r)|^2}{|\tilde{\xi}|^{2H_0-1}}\bigg)\, . 
\end{align*}
Using again a hyperspherical change of variable (with respect to $\eta$), we obtain
\begin{align*}
\ca_2&\lesssim \int_{0}^{\infty}\frac{d\rho}{\rho^{4(H_1+...+H_d)-4d+1}}\int_{\frac{1}{2}<r<\frac{3}{2}}\frac{dr}{\{1+\rho^2(1-r)^2\}^{2\alpha}} \nonumber \\
&\hspace{2cm}\bigg(\int_{\mathbb{R}}d\xi\, \frac{|\gamma_t(\xi,\rho)|^2}{|\xi|^{2H_0-1}}\bigg)^{\frac{1}{2}}\cdot \bigg(\int_{\mathbb{R}}d\xi\,  \frac{|\gamma_{s,t}(\xi,\rho)|^2}{|\xi|^{2H_0-1}}\bigg)^{\frac{1}{2}}\cdot \bigg(\int_{\mathbb{R}}d\tilde{\xi}\, \frac{|\gamma_t(\tilde{\xi},\rho r)|^2}{|\tilde{\xi}|^{2H_0-1}}\bigg)\, , \nonumber 
\end{align*}
and we can now use Corollary \ref{tec} to assert that 
\begin{align*}
\ca_2&\lesssim |t-s|^{\varepsilon}\Bigg[\int_{0}^{1}\frac{d\rho}{\rho^{4(H_1+...+H_d)-4d+1}}\\
&\hspace{2cm}+\int_{1}^{\infty}\frac{d\rho}{\rho^{4(2H_0+H_1+...+H_d)-4d+1-8\varepsilon}}\int_{\frac{1}{2}<r<\frac{3}{2}}\frac{dr}{\{1+\rho^2(1-r)^2\}^{2\alpha}}\Bigg]
\end{align*}
for all $0<\varepsilon<H_0$. 

\smallskip

Now, remark that
\begin{align}
&\int_{1}^{\infty}\frac{d\rho}{\rho^{4(2H_0+H_1+...+H_d)-4d+1-8\varepsilon}}\int_{\frac{1}{2}<r<\frac{3}{2}}\frac{dr}{\{1+\rho^2(1-r)^2\}^{2\alpha}} \nonumber \\
&\leq \int_{1}^{\infty}\frac{d\rho}{\rho^{4\alpha+4(2H_0+H_1+...+H_d)-4d+1-8\varepsilon}}\int_{\frac{1}{2}<r<\frac{3}{2}}\frac{dr}{(1-r)^{4\alpha}} \, .\label{cond-al-imp}
\end{align}
Thanks to our hypothesis on $\al$ (see \eqref{assump-al-bis}), we know that $4\al <1$ and we can choose $\varepsilon>0$ such that 
$$4\alpha+4(2H_0+H_1+...+H_d)-4d+1-8\varepsilon>1\, . $$
For such a choice of parameter, the two integrals in the right-hand side of \eqref{cond-al-imp} are clearly finite, and finally we have proved that
$$\ca_2\lesssim |t-s|^{\varepsilon} \, .$$

\smallskip

Going back to \eqref{boun-cj-1-m}, we have thus shown that, uniformly over $m$,
$$\cj^1_{m;s,t} \lesssim |t-s|^{\varepsilon} \, .$$ 
It is now easy to realize that the other three integrals $\{\cj^i_{m;s,t},\, i=2,3,4\}$ (as defined in \eqref{defi-cj-i}) could be controlled with the very same arguments (with the very same resulting bound). 

\smallskip

Injecting the above estimates into \eqref{bou-sum-cj-i-p} provides us with the desired conclusion \eqref{mom-estim-wick-sq}.

\

\noindent
\textbf{Step 2: Conclusion.} Let $\al$ satisfying \eqref{assump-gene-al}, that is $\alpha > d-\big(2H_0+\sum_{i=1}^{d}H_i\big)$.  

\smallskip

If $\al < \frac14$, then condition \eqref{assump-al-bis} is satisfied, and so the moment estimate \eqref{mom-estim-wick-sq} holds true. Thanks to this bound, we can reproduce the arguments used in Step 2 of Section \ref{subsec:proof-sto} to get that $(\chi^2\<Psi2>_{n})_{n\geq 1}$ converges almost surely in the space $\mathcal{C}([0,T]; \mathcal{W}^{-2\alpha,p}(\mathbb{R}^d))$, for all $2\leq p <\infty$.

\smallskip

If $\al \geq \frac14$, remark that, due to assumption \eqref{cond-hurst-psi-2}, we can choose $\al'$ verifying $\al'<\al$ and $d-\big(2H_0+\sum_{i=1}^{d}H_i\big) <\al'<\frac14$ (that is, $\al'$ satisfies \eqref{assump-al-bis}). By executing the above scheme again, we deduce that the sequence $(\chi^2\<Psi2>_{n})_{n\geq 1}$ converges almost surely in $\mathcal{C}([0,T]; \mathcal{W}^{-2\alpha',p}(\mathbb{R}^d))$, and therefore it converges almost surely in $\mathcal{C}([0,T]; \mathcal{W}^{-2\alpha,p}(\mathbb{R}^d))$ as well, for all $2\leq p<\infty$.

\smallskip

Finally, the (a.s.) convergence in $\mathcal{C}([0,T]; \mathcal{W}^{-2\alpha,\infty}(\mathbb{R}^d))$ is the result of the Sobolev embedding $\cw^{-2\al+\frac{d}{p}+\eta,p}(\R^d)\subset \cw^{-2\al,\infty}(\R^d)$, for any $\eta>0$, which ends the proof of Proposition \ref{sto1}.
\end{proof}

\subsection{Asymptotic estimate of the renormalization constant}

\

\smallskip

Fix $d\geq 1$ and $(H_0,...,H_d) \in (0,1)^{d+1}$ such that 
$$2H_0+\sum_{i=1}^{d}H_i\leq d\, .$$ 
The objective of this subsection is to give an equivalent of the quantity $\sigma_n(t,x)\ =\ \mathbb{E}[\<Psi>_n(t,x)^2]$.
Let us rewrite it under the integral form:
\begin{eqnarray*}
\sigma_n(t,x)\ =\ \mathbb{E}[\<Psi>_n(t,x)^2]&=&c^2\int_{|\xi|\leq 4^n} \frac{d\xi}{|\xi|^{2H_0-1}}\int_{|\eta|\leq 2^n}\prod\limits_{i=1}^{d}\frac{d\eta_i}{|\eta_i|^{2H_i-1}}|\gamma_t(\xi,|\eta|)|^2 \\
&=&C\int_{0}^{2^n} \frac{dr}{r^{2(H_1+...+H_d)-2d+1}}\int_{|\xi|\leq 4^n}\frac{d\xi}{|\xi|^{2H_0-1}}|\gamma_t(\xi,r)|^2 \, ,
\end{eqnarray*}
where the previous identity is obtained after a hyperspherical change of variables. As remarked in Proposition \ref{prop:renorm-cstt}, the above formula shows the surprising fact that $\sigma_n$ does not depend on $x$.

\smallskip

The desired estimate \eqref{estim-sigma-n} is now a consequence of the following technical result (applied with $\alpha:=2H_0 \in (0,2)$ and $\kappa:=d-[2H_0+\sum_{i=1}^{d}H_i] \geq 0$):

\begin{proposition}\label{equi}
Fix $t>0$. For all $\alpha \in (0,2)$ and $\kappa \geq 0$, there exists two constants $c_1$ and $c_2$ depending only on $\al$ and $\kappa$ such that
$$\int_{0}^{2^n}\frac{dr}{r^{-2\alpha-2\kappa+1}}\int_{|\xi|\leq 4^n}\frac{d\xi}{|\xi|^{\alpha-1}}|\gamma_t(\xi,r)|^2\underset{n \rightarrow \infty}\sim  \left\{
    \begin{array}{ll}
        c_1 4^{n \kappa} & \mbox{if}\hspace{0,3cm} \kappa>0 \,   \\
        c_2 n  & \mbox{if}\hspace{0,3cm} \kappa=0\, 
    \end{array}
\right. .$$
\end{proposition}

\begin{proof}
It is quite easy to verify from the expression
\begin{equation}
\gamma_t(\xi,r):=e^{i\xi t}\int_{0}^{t}e^{-s r^2}e^{-i \xi s }ds \, 
\end{equation}
that
$$|\gamma_t(\xi,r)|^2=\frac{1-2\cos(\xi t)e^{-r^2 t}+e^{-2r^2 t}}{r^4+\xi^2}\, .$$ 
With the change of variables $u=r^2$ and using the parity in $\xi$ of the function $\frac{|\gamma_t(\xi,r)|^2}{|\xi|^{\alpha-1}},$
we rewrite the integral we are dealing with as
\begin{align}
\int_{0}^{2^n}\frac{dr}{r^{-2\alpha-2\kappa+1}}\int_{|\xi|\leq 4^n}\frac{d\xi}{|\xi|^{\alpha-1}}|\gamma_t(\xi,r)|^2 &=
\int_{0}^{4^n}\frac{dr}{r^{-\alpha-\kappa+1}}\int_{0}^{4^n}\frac{d\xi}{{\xi}^{\alpha-1}}\frac{1-2\cos(\xi t)e^{-r t}+e^{-2r t}}{r^2+\xi^2}\nonumber \\
&= 4^{n\kappa}\int_{0}^{1}\frac{dr}{r^{-\alpha-\kappa+1}}\int_{0}^{1}\frac{d\xi}{\xi^{\alpha-1}}\frac{1-2\cos( 4^n \xi t)e^{-4^n r t}+e^{-2.4^n r t}}{r^2+\xi^2}.\nonumber 
\end{align}
We perform the change of variables $(r,\xi) \mapsto T(r,\xi)$ defined by $$ T : (r,\xi) \mapsto T(r,\xi)=\Big(\frac{r}{\xi},\xi\Big).$$
It is readily checked that $T$ is a one-to-one map of $(0,1)^2$ on $$\{(x,y)\in \mathbb{R}^2, x>0, 0<y<\min\Big(1,\frac{1}{x}\Big)\}.$$
Moreover, its reverse is explicitly given by $T^{-1}: (x,y) \mapsto (xy,y)$ whose absolute value of the Jacobian equals $y$.
\begin{align}
\int_{0}^{2^n}\frac{dr}{r^{-2\alpha-2\kappa+1}}\int_{|\xi|\leq 4^n}\frac{d\xi}{|\xi|^{\alpha-1}}&|\gamma_t(\xi,r)|^2
= 4^{n\kappa}\int_0^1 \frac{x^{\al+\kappa-1}}{1+x^2}dx\int_0^1\frac{dy}{y^{1-\kappa}}\Big(1-2\cos( 4^n y t)e^{-4^n xy t}+e^{-2.4^n xy t}\Big)\nonumber \\
& \hspace{0,1cm}+4^{n\kappa}\int_1^{+\infty} \frac{x^{\al+\kappa-1}}{1+x^2}dx\int_0^{\frac{1}{x}}\frac{dy}{y^{1-\kappa}}\Big(1-2\cos( 4^n y t)e^{-4^n xy t}+e^{-2.4^n xy t}\Big). \nonumber
\end{align}
\noindent
\textit{First case: $\ka>0$}.
By Lebesgue's dominated convergence theorem, we observe that
\begin{align}
& \int_0^1 \frac{x^{\al+\kappa-1}}{1+x^2}dx\int_0^1\frac{dy}{y^{1-\kappa}}\Big(-2\cos( 4^n y t)e^{-4^n xy t}+e^{-2.4^n xy t}\Big)\nonumber \\
& \hspace{0,1cm}+\int_1^{+\infty} \frac{x^{\al+\kappa-1}}{1+x^2}dx\int_0^{\frac{1}{x}}\frac{dy}{y^{1-\kappa}}\Big(-2\cos( 4^n y t)e^{-4^n xy t}+e^{-2.4^n xy t}\Big)\underset{n \rightarrow \infty}\rightarrow 0, \nonumber
\end{align}
leading to
$$\int_{0}^{2^n}\frac{dr}{r^{-2\alpha-2\kappa+1}}\int_{|\xi|\leq 4^n}\frac{d\xi}{|\xi|^{\alpha-1}}|\gamma_t(\xi,r)|^2\underset{n \rightarrow \infty}\sim c_1 4^{n \kappa},$$
where 
\begin{eqnarray*}
c_1&=&\int_0^1 \frac{x^{\al+\kappa-1}}{1+x^2}dx\int_0^1\frac{dy}{y^{1-\kappa}}+
\int_1^{+\infty} \frac{x^{\al+\kappa-1}}{1+x^2}dx\int_0^{\frac{1}{x}}\frac{dy}{y^{1-\kappa}}\\
&=&\frac{1}{\kappa}\Bigg(\int_0^1 \frac{x^{\al+\kappa-1}}{1+x^2}dx+\int_1^{+\infty} \frac{x^{\al-1}}{1+x^2}dx\Bigg).
\end{eqnarray*}
A more visual expression of $c_1$ can be found in the appendix.\\
\textit{Second case: $\ka=0$.} Let us recall that
\begin{align}
\int_{0}^{2^n}\frac{dr}{r^{-2\alpha+1}}\int_{|\xi|\leq 4^n}\frac{d\xi}{|\xi|^{\alpha-1}}|\gamma_t(\xi,r)|^2 &=\int_0^1 \frac{x^{\al-1}}{1+x^2}dx\int_0^1\frac{dy}{y}\Big(1-2\cos( 4^n y t)e^{-4^n xy t}+e^{-2.4^n xy t}\Big)\nonumber \\
& \hspace{0,1cm}+\int_1^{+\infty} \frac{x^{\al-1}}{1+x^2}dx\int_0^{\frac{1}{x}}\frac{dy}{y}\Big(1-2\cos( 4^n y t)e^{-4^n xy t}+e^{-2.4^n xy t}\Big). \nonumber
\end{align}
Let us introduce the function $f$ defined for all $T>0$ by
\begin{align}
f(T)&=\int_0^1 \frac{x^{\al-1}}{1+x^2}dx\int_0^1\frac{dy}{y}\Big(1-2\cos( T y )e^{-T xy }+e^{-2T xy }\Big)\nonumber \\
& \hspace{0,1cm}+\int_1^{+\infty} \frac{x^{\al-1}}{1+x^2}dx\int_0^{\frac{1}{x}}\frac{dy}{y}\Big(1-2\cos( T y )e^{-T xy }+e^{-2T xy }\Big).\nonumber
\end{align}
Then, one has by derivation
\begin{align}
f'(T)&=\int_0^1 \frac{x^{\al-1}}{1+x^2}dx\int_0^1 dy\Big(2\sin( T y )e^{-T xy }+2x\cos(Ty)e^{-T xy }-2xe^{-2T xy }\Big)\nonumber \\
& \hspace{0,1cm}+\int_1^{+\infty} \frac{x^{\al-1}}{1+x^2}dx\int_0^{\frac{1}{x}}dy\Big(2\sin( T y )e^{-T xy }+2x\cos(Ty)e^{-T xy }-2xe^{-2T xy }\Big).\nonumber
\end{align}
To deal with this derivative with ease, we will resort to the following technical lemma:

\begin{lemma}\label{teco}
Let $a$, $x$ and $T$ three positive numbers. It holds that:
\begin{eqnarray*}
\int_0^a \sin( T y )e^{-T xy }dy &=&-\frac{x e^{-Tax}\sin(Ta)}{(1+x^2)T}+\frac{1-e^{-Tax}\cos(Ta)}{(1+x^2)T};\\
\int_0^a \cos(Ty)e^{-T xy }dy&=&\frac{ e^{-Tax}\sin(Ta)}{(1+x^2)T}+
\frac{x(1-e^{-Tax}\cos(Ta))}{(1+x^2)T};\\
\int_0^a e^{-2Txy}dy&=&\frac{1-e^{-2Txa}}{2Tx}.
\end{eqnarray*}
\end{lemma}

\noindent
Lemma \ref{teco} with $a=1$ and $a=\frac{1}{x}$ readily entails:
\begin{align}
f'(T)&=\frac{1}{T}\Bigg(\int_0^{+\infty} \frac{x^{\al-1}}{1+x^2}dx+\int_0^{1} \frac{x^{\al-1}}{1+x^2}dx\big(e^{-2Tx}-2\cos(T)e^{-Tx}\big) \nonumber\\
& \hspace{0,1cm}+\int_1^{+\infty} \frac{x^{\al-1}}{1+x^2}dx\big(e^{-2T}-2\cos\Big(\frac{T}{x}\Big)e^{-T}\big)\Bigg).\nonumber
\end{align}
By Lebesgue's dominated convergence theorem, we deduce since $\al \in (0,2)$ that
$$\int_0^{1} \frac{x^{\al-1}}{1+x^2}dx\big(e^{-2Tx}-2\cos(T)e^{-Tx}\big)+\int_1^{+\infty} \frac{x^{\al-1}}{1+x^2}dx\big(e^{-2T}-2\cos\Big(\frac{T}{x}\Big)e^{-T}\big)\underset{T \rightarrow \infty}\rightarrow 0$$
and consequently
$$f'(T)\underset{T \rightarrow \infty}\sim \frac{c_2}{T}\, $$
where $c_2=\displaystyle \int_0^{+\infty} \frac{x^{\al-1}}{1+x^2}dx.$
For the sake of beauty, let us compute the value of the constant $c_2$.
We recall the classical result, coming from complex analysis, stating that for all $\beta \in ]0,1[,$
$$\int_{0}^{+\infty} \frac{d t}{t^{\beta}(1+t)}=\frac{\pi}{\sin \pi \beta}.$$
It yields:
$$f'(T)\underset{T \rightarrow \infty}\sim \frac{\pi}{2\sin(\frac{\al \pi}{2})T}\, .$$
We can finally use a standard comparison argument (see Lemma \ref{int} below) to assert that
$$f(T)\underset{T \rightarrow \infty}\sim \frac{\pi}{2\sin(\frac{\al \pi}{2})}\ln(T)  .$$
We obtain the equivalent we are looking for, namely
$$\int_{0}^{2^n}\frac{dr}{r^{-2\alpha+1}}\int_{|\xi|\leq 4^n}\frac{d\xi}{|\xi|^{\alpha-1}}|\gamma_t(\xi,r)|^2\underset{n \rightarrow \infty}\sim \frac{\pi \ln(2)}{\sin(\frac{\al \pi}{2})}n .$$
\end{proof}

\begin{lemma}\label{int}
Fix $a \in \mathbb{R}$ and let $g:[a,+\infty[\rightarrow \mathbb{R}$, $h:[a,+\infty[\rightarrow (0,\infty)$, be two continuous functions. If $g(t)\underset{t \rightarrow \infty}\sim h(t)$ and $\int_a^{+\infty}h(t)dt=\infty$, then
$$\int_a^{T}g(t)dt\underset{T \rightarrow \infty}\sim \int_a^{T}h(t)dt \, .$$
\end{lemma}

\subsection{Details about the definition of the linear solution}\label{subsec:glob-def}

\

\smallskip

As in \cite{deya3}, the statement of Proposition~\ref{sto} provides us with a \emph{local} definition of the process $\<Psi>$, that is, up to multiplication by $\chi \in \cac_c^\infty(\R^d)$. To put it differently, what is actually given by the proposition is the set of the limit elements $\{\chi\<Psi>, \, \chi\in \cac_c^\infty(\R^d)\}$. Let us briefly recall how those elements can be gathered into a single process $\<Psi>$.

\smallskip

Fix $p\geq 2$ and $\al$ verifying~\eqref{assump-gene-al}. We denote by $\cp$ the set of sequences $\si=(\si_k)_{k\geq 1}$ such that for every $k\geq 1$, $\si_k:\R^d \to \R$ is a smooth function verifying
\begin{equation*}
    \si_k(x)=  \left\{
    \begin{array}{ll}
		1 & \mbox{if}\  \|x\|\leq k \, ,\\
        0 & \mbox{if} \ \|x\|\geq k+1\, .  
    \end{array}
\right.
\end{equation*}

Let us fix such a sequence $\si$, and for all $k\geq 1$, we call $\<Psi>^{(\si_k)}$ the limit of the sequence $(\si_k\<Psi>_n)_{n\geq 1}$ in the space $\cac([0,T];\cw^{-\al,p}(\R^d))$, as given by Proposition~\ref{sto}. As $\<Psi>^{(\si_k)}$ is defined on a probability space $\Omega^{(\si_k)}$ of full measure $1$, $\Omega^{(\si)}:=\cap_{k\geq 1}\Omega^{(\si_k)}$ is still of measure $1$.

\smallskip

For each time $t \in [0,T]$, we are now able to define the random distribution 
$$\<Psi>^{(\si)}(t):\Omega^{(\si)}\to \cd'(\R^d)$$
as follows: for all smooth compactly-supported function $\vp:\R^d \to \R$ such that $\textup{supp}(\vp) \subset B(0,k)$ (for some $k\geq 1$),
\begin{equation*}
\langle \<Psi>^{(\si)}(t), \vp \rangle := \langle \<Psi>^{(\si_k)}(t), \vp \rangle \, .
\end{equation*}

\begin{proposition}\label{prop:defi-psi-glo}
\begin{enumerate}
\item The above distribution $\<Psi>^{(\si)}$ is well defined, i.e. for every $1\leq k\leq \ell$ and for all smooth compactly-supported function $\vp:\R^d \to \R$ with $\textup{supp}(\vp) \subset B(0,k)\subset B(0,\ell)$, one has
$$\langle \<Psi>^{(\si_k)}(t), \vp \rangle =\langle \<Psi>^{(\si_\ell)}(t), \vp \rangle  \quad \text{on} \ \ \Omega^{(\si)} \, .$$
\item For any smooth compactly-supported function $\chi:\R^d \to \R$, one has, on $\Omega^{(\si)}$,
$$\chi \cdot\<Psi>_n \underset{n \rightarrow \infty}\rightarrow \chi\cdot \<Psi>^{(\si)} \quad \text{in}\ \ \mathcal{C}([0,T]; \mathcal{W}^{-\al,p}(\mathbb{R}^d)) \, .$$
\item If $\si,\ga \in \cp$, it holds that
$$\<Psi>^{(\si)}=\<Psi>^{(\ga)} \quad \text{on} \ \ \Omega^{(\si)}\cap \Omega^{(\ga)} \, .$$
\end{enumerate}
Thanks to the latter identification property, we can set $\<Psi>:=\<Psi>^{(\si)}$, as soon as $\si\in \cp$ is a fixed element.
\end{proposition}

\begin{remark}
The latter procedure can of course be used to give a rigorous meaning to the second-order process $\chi^2 \<Psi2>$ as a distribution-valued function $\<Psi2>$.
\end{remark}

\section{Rougher stochastic constructions when the space dimension equals 2}\label{rough}

The aim of this section is to establish the proofs of Proposition \ref{prop:stoch-constr} and Proposition \ref{prop:limite-case}, that is to construct the second order stochastic process $\<Psi2>$ in the roughest case, namely when $\frac{3}{2} < 2H_0+H_1+H_2 \leq  \frac74 $, and to show that this construction is in some sense optimal insofar as 
$$\mathbb{E}\Big[ \|\chi \cdot \<Psi2>^n(t,.)\|_{H^{-2\al}}^2\Big]\underset{n\to +\infty}{\longrightarrow} \ +\infty,$$
for all compactly-supported function $\chi$ and $t>0$ when $2H_0+H_1+H_2 \leq  \frac32.$

\subsection{Additional notations}

To begin with, let us introduce some notations that will be frequently used in the proofs of Proposition \ref{prop:stoch-constr} and Proposition \ref{prop:limite-case}.
For all $\tau\in \big\{\<Psi>,\<Psi2> \big\}$, $0\leq n \leq m$ and  $0\leq s,t\leq 1$, let us set $\tau^{n,m} := \tau^m-\tau^n$, and then, for $f\in \{\tau^{n},\tau^{m},\tau^{n,m}\}$, $f_{s,t} := f_t-f_s$.

\smallskip
Then, we set for all $H=(H_0,H_1,H_2)\in (0,1)^3$, $\eta \in \R^2$,
\begin{equation}\label{n-h}
K^H(\eta):= \frac{|\eta_1|^{1-2H_1} |\eta_2|^{1-2H_2}}{1+|\eta|^{4H_0}} \ .
\end{equation}
For all $a=(a_1,a_2)$, resp. $b=(b_1,b_2)$, with $a_i\in \{n,m,\{n,m\}\}$, resp. $b_i\in \{s,t,\{s,t\}\}$, we write
$$L^{H,a}_b(\eta) := \frac{1}{|\eta_1|^{2H_1-1} |\eta_2|^{2H_2-1}} \int_{(\xi,\eta)\in D^{a_1} \cap D^{a_2}} d\xi \, \frac{\ga_{b_1}(\xi,|\eta|) \overline{\ga_{b_2}(\xi,|\eta|)}}{|\xi|^{2H_0-1}}  $$
where
$$D_n:=B_{2n}^1 \times B_n^2 \quad \text{with} \quad B_\ell^k:=\left\{ \lambda \in \mathbb{R}^k: |\lambda| \leq 2^\ell \right\} \quad \text{and} \quad D^{n,m}:= D^m \backslash D^n  \ ,$$
in such a way that for all $y,\yti\in \R^2$, 
\begin{equation}\label{cova-gene}
\mathbb{E}\big[  \<Psi>^{a_1}_{b_1}(y)\overline{\<Psi>^{a_2}_{b_2}(\tilde{y})}\big]=c_H ^2\int_{\R^2} d\eta \, e^{\imath \langle \eta,y\rangle}e^{-\imath \langle \eta,\yti\rangle} L^{H,a}_b(\eta) \ .
\end{equation}
In the following, we will resort to several technical lemmas whose statements and proofs can be found in the Appendix. In particular, Lemma \ref{lem:handle-chi-correc} is of major interest since it describes in some way the role of the cut-off function $\chi$ which allows a gain of integrability.

\subsection{Proof of Proposition \ref{prop:stoch-constr}}

The strategy is exactly the same as the one developed in the proof of Proposition \ref{sto} and results from the combination of Kolmogorov criterion and Garsia-Rodemich-Rumsey lemma. Let us write
\begin{align}
&\mathbb{E}\bigg[\Big|\mathcal{F}^{-1}\Big(\{1+|.|^2\}^{-\al}\mathcal{F}\big(\chi^2 \cdot \<Psi2>^{n,m}_{s,t}\big)\Big)(x)\Big|^{2}\bigg]\nonumber\\
&=\frac{1}{(2\pi)^{4}}\iint_{(\R^2)^2} \frac{d\la d\lati}{\{1+|\la|^2\}^{\al} \{1+|\lati|^2\}^{\al}}e^{\imath \langle x,\la-\lati\rangle} \iint_{(\R^2)^2} d\xi d\xiti \, \widehat{\chi^2}(\la-\xi) \overline{\widehat{\chi^2}(\lati-\xiti)} \mathcal{Q}_{n,m;s,t}(\xi,\xiti) \, ,\label{bound-correc-q}
\end{align}
with
\begin{eqnarray*}
\mathcal{Q}_{n,m;s,t}(\xi,\xiti)&:=&\mathbb{E}\Big[\cf\big(\<Psi2>^{n,m}_{s,t}\big)(\xi) \overline{\cf\big(\<Psi2>^{n,m}_{s,t}\big)(\xiti)}\Big]\\
&=&\ \iint_{(\R^2)^2} dy d\yti \, e^{-\imath \langle \xi,y \rangle}e^{\imath \langle \tilde{\xi},\yti\rangle} \mathbb{E}\big[  \<Psi2>^{n,m}_{s,t}(y)\overline{\<Psi2>^{n,m}_{s,t}(\tilde{y})}\big].
\end{eqnarray*}
\noindent
Resorting to Wick formula, we derive 
\begin{eqnarray*}\label{exp-wick}
\mathbb{E}\big[  \<Psi2>^{n,m}_{s,t}(y)\overline{\<Psi2>^{n,m}_{s,t}(\tilde{y})}\big]&=&\sum_{(a,b)\in \ca}c_{a,b}\, \mathbb{E}\big[  \<Psi>^{a_1}_{b_1}(y)\overline{\<Psi>^{a_2}_{b_2}(\tilde{y})}\big] \mathbb{E}\big[  \<Psi>^{a_3}_{b_3}(y)\overline{\<Psi>^{a_4}_{b_4}(\tilde{y})}\big]\\
&=&\sum_{(a,b)\in \ca}\tilde{c_{a,b}}\, \iint_{(\R^2)^2} d\eta d\etati \,  e^{\imath \langle \eta,y\rangle}e^{-\imath \langle \eta,\yti\rangle}  e^{\imath \langle \etati,y\rangle}e^{-\imath \langle \etati,\yti\rangle} L^{H,(a_1,a_2)}_{b_1,b_2}(\eta)L^{H,(a_3,a_4)}_{b_3,b_4}(\etati) \ ,
\end{eqnarray*}
for some index set $\ca$ such that $a_i\in \{n,m,\{n,m\}\}$, $b_i\in \{s,t,\{s,t\}\}$, and one has both $\{a_1,\ldots,a_4\} \cap \{\{n,m\}\}\neq \emptyset$ and $\{b_1,\ldots,b_4\} \cap \{\{s,t\}\}\neq \emptyset$.
It leads us to
\begin{align}
   \mathcal{Q}_{n,m;s,t}(\xi,\xiti)&=\sum_{(a,b)\in \ca}\tilde{c_{a,b}} \iint_{(\R^2)^2} d\eta d\etati \,  L^{H,(a_1,a_2)}_{b_1,b_2}(\eta)L^{H,(a_3,a_4)}_{b_3,b_4}(\etati)\delta_{\{\xi=\eta+\etati\}} \delta_{\{\xiti=\eta+\etati\}}  \ . \nonumber
\end{align}
Consequently,
$$\mathbb{E}\bigg[\Big|\mathcal{F}^{-1}\Big(\{1+|.|^2\}^{-\al}\mathcal{F}\big(\chi^2 \cdot \<Psi2>^{n,m}_{s,t}\big)\Big)(x)\Big|^{2}\bigg]=\sum_{(a,b)\in \ca} \phi_{a,b},$$
with
\begin{align}
  \phi_{a,b}&= \tilde{c_{a,b}} \iint_{(\R^2)^2} \frac{d\la d\lati}{\{1+|\la|^2\}^{\al} \{1+|\lati|^2\}^{\al}}e^{\imath \langle x,\la-\lati\rangle}\nonumber \\ 
  &\hspace{1cm}\iint_{(\R^2)^2} d\eta d\etati L^{H,(a_1,a_2)}_{b_1,b_2}(\eta)L^{H,(a_3,a_4)}_{b_3,b_4}(\etati)\widehat{\chi^2}(\la-(\eta+\etati)) \overline{\widehat{\chi^2}(\lati-(\eta+\etati))}.
\end{align}
With the help of Lemma~\ref{lem:handle-chi-correc} and the hypercontractivity of Gaussian chaoses, we obtain
\begin{align*}
\int_{\mathbb{R}^2}dx\, \mathbb{E}\bigg[\Big|\mathcal{F}^{-1}\Big(\{1+|.|^2\}^{-\al}\mathcal{F}\big(\chi^2 \cdot \<Psi2>^{n,m}_{s,t}\big)\Big)(x)\Big|^{2p}\bigg]\lesssim  \left(\sum_{(a,b)\in \ca} \psi_{a,b}\right)^p \, ,
\end{align*}
where 
\begin{align}
  \psi_{a,b}&=  \iint_{(\R^2)^2}  d\eta d\etati \, \{1+|\eta+\etati|^2\}^{-2\alpha} |L^{H,(a_1,a_2)}_{b_1,b_2}(\eta)||L^{H,(a_3,a_4)}_{b_3,b_4}(\etati)|. 
\end{align}
Now, for $0<\varepsilon<H_0$, Lemma \ref{lem:bou-l-h} provides us with the bound
\begin{align}
    \psi_{a,b}&\lesssim 2^{-2n\varepsilon}|t-s|^{\varepsilon} \iint_{(\R^2)^2}  d\eta d\etati \, \{1+|\eta+\etati|^2\}^{-2\alpha}\nonumber\\
    &\hspace{0,8cm}\big\{ K^{H_{\varepsilon,0}}(\eta)+K^{H_{\varepsilon,0,1}}(\eta)+K^{H_{\varepsilon,0,2}}(\eta) \big\}\big\{ K^{H_{\varepsilon,0}}(\etati)+K^{H_{\varepsilon,0,1}}(\etati)+K^{H_{\varepsilon,0,2}}(\etati) \big\}.
\end{align}

According to Lemma \ref{lem:tech-ordre-deux}, the latter quantity is finite as soon as $\varepsilon$ is small enough and we have finally obtained:
$$\int_{\mathbb{R}^2}dx\, \mathbb{E}\bigg[\Big|\mathcal{F}^{-1}\Big(\{1+|.|^2\}^{-\al}\mathcal{F}\big(\chi^2 \cdot \<Psi2>^{n,m}_{s,t}\big)\Big)(x)\Big|^{2p}\bigg]\lesssim 2^{-2n\varepsilon p}|t-s|^{ \varepsilon p}.$$
We can mimic the arguments at the end of the proof of Proposition \ref{sto} to get the result.

\subsection{Proof of Proposition \ref{prop:limite-case}}
Suppose that $d\geq1$ and that $(H_0,H_1,...,H_d)\in (0,1)^{d+1}$ verifies the condition $2H_0+H_1+\cdots+H_d \leq \frac{3}{4}d.$
As usual, we will use the notation $A\gtrsim B$ whenever one can find a constant $c>0$ such that $A\geq c B$. Let us also introduce the additional notation
\begin{equation}
\gga^{H_0,n}_{t}(r):=\int_{-4^n}^{4^n} d\xi \, \frac{|\ga_t(\xi,r)|^2 }{|\xi|^{2H_0-1}} \ .
\end{equation}

\smallskip

\noindent
Fix $t>0$. Using (\ref{cova-gene}) and then Wick formula, we get that
\begin{eqnarray*}
\lefteqn{\mathbb{E}\Big[ \|\chi \cdot \<Psi2>^n(t,.)\|_{H^{-2\al}}^2\Big]} \\
&=&c\int_{|\eta|\leq 2^n} d\eta \int_{|\etati|\leq 2^n} d\etati\, \prod_{i=1}^{d}\frac{1}{|\eta_i|^{2H_i-1}} \prod_{i=1}^{d}\frac{1}{|\etati_i|^{2H_i-1}} \gga^{H_0,n}_{t}(|\eta|) \gga^{H_0,n}_{t}(|\etati|)  \int_{\R^d} \frac{d\xi}{\{1+|\xi|^2\}^{2\al}} \big| \widehat{\chi}(\xi-(\eta-\etati))\big|^2 \, .
\end{eqnarray*}
With the change of variables $\tilde{\xi}:=\xi-(\eta-\etati)$, it holds that
\begin{align*}
&\int_{\R^d} \frac{d\xi}{\{1+|\xi|^2\}^{2\al}} \big| \widehat{\chi}(\xi-(\eta-\etati))\big|^2\\
&= \int_{\R^d} \frac{d\xi}{\{1+|\xi+(\eta-\etati)|^2\}^{2\al}} \big| \widehat{\chi}(\xi)\big|^2 \gtrsim \frac{1}{\{1+|\eta-\etati|^2\}^{2\al}}\int_{\R^d}\frac{d\xi}{\{1+|\xi|^2\}^{2\al}} \big| \widehat{\chi}(\xi)\big|^2  \, .
\end{align*}
As $\chi$ is a non-zero compactly-supported function, the support of $\widehat{\chi}$ contains a ball of radius $R>0$ that guarantees $\displaystyle\int_{\R^d}\frac{d\xi}{\{1+|\xi|^2\}^{2\al}} \big| \widehat{\chi}(\xi)\big|^2>0.$ Performing the changes of variables $r_1=\frac{\eta_1-\etati_1}{\eta_1},\cdots, r_d=\frac{\eta_d-\etati_d}{\eta_d}$ (second inequality) followed by $\tilde{r_1}=\eta_1 r_1,\cdots, \tilde{r_d}=\eta_d r_d$ (third inequality), we write
\begin{eqnarray*}
\lefteqn{\mathbb{E}\Big[ \|\chi \cdot \<Psi2>^n(t,.)\|_{H^{-2\al}}^2\Big]}\\
&\gtrsim&\int_{0}^{\frac{2^{n}}{\sqrt{d}}} d\eta_1 \int_{\frac12 \eta_1}^{\eta_1} d\etati_1 \cdots \int_0^{\frac{2^{n}}{\sqrt{d}}} d\eta_d \int_{\frac12 \eta_d}^{\eta_d} d\etati_d \,   \frac{\gga^{H_0,n}_{t}(|\eta|) \gga^{H_0,n}_{t}(|\etati|)}{\{1+|\eta-\etati|^2\}^{2\al}}  \prod_{i=1}^{d}\frac{1}{|\eta_i|^{2H_i-1}} \prod_{i=1}^{d}\frac{1}{|\etati_i|^{2H_i-1}}  \\
&\gtrsim& \int_{0}^{\frac{2^{n}}{\sqrt{d}}}\cdots \int_0^{\frac{2^{n}}{\sqrt{d}}}d\eta_1\cdots d\eta_d \prod_{i=1}^{d}\frac{1}{|\eta_i|^{4H_i-3}}\int_0^{\frac12}\cdots \int_0^{\frac12}\frac{dr_1\cdots dr_d}{\{1+\eta_1^2r_1^2+\cdots+\eta_d^2r_d^2\}^{2\al}} \\
& &\hspace{6cm} \gga^{H_0,n}_{t}(|\eta|) \gga^{H_0,n}_{t} \Big( \sqrt{\eta_1^2 (1-r_1)^2+\cdots+\eta_d^2(1-r_d)^2}\Big)\\
&\gtrsim& \int_{0}^{\frac{2^{n}}{\sqrt{d}}}\cdots \int_0^{\frac{2^{n}}{\sqrt{d}}}d\eta_1\cdots d\eta_d \prod_{i=1}^{d}\frac{1}{|\eta_i|^{4H_i-2}}\int_0^{\frac{1}{2}\eta_1}\cdots \int_0^{\frac{1}{2}\eta_d}\frac{dr_1\cdots dr_d}{\{1+r_1^2+\cdots+r_d^2\}^{2\al}} \\
& &\hspace{6cm} \gga^{H_0,n}_{t}(|\eta|) \gga^{H_0,n}_{t} \Big( \sqrt{\eta_1^2 \bigg(1-\frac{r_1}{\eta_1}\bigg)^2+\cdots+\eta_d^2\bigg(1-\frac{r_d}{\eta_d}\bigg)^2}\Big).
\end{eqnarray*}
The hyperspherical change of variables below\\
$$
\left\{
    \begin{array}{ll}
        \ |\eta|=r \\
        \ \eta_1=r\cos(\theta_1)\\
        \ \eta_2=r\sin(\theta_1)\cos(\theta_2)\\
        \ .\\
        \ .\\
        \ .\\
        \ \eta_{d-1}=r\sin(\theta_1)\cdots\sin(\theta_{d-2})\cos(\theta_{d-1})\\
        \ \eta_{d}=r\sin(\theta_1)\cdots\sin(\theta_{d-2})\sin(\theta_{d-1})\\
    \end{array}
\right.
$$
whose absolute value of the Jacobian equals $r^{d-1}\prod\limits_{i=1}^{d}|\sin(\theta_i)|^{d-1-i}$ entails that
\begin{eqnarray*}
\lefteqn{\mathbb{E}\Big[ \|\chi \cdot \<Psi2>^n(t,.)\|_{H^{-2\al}}^2\Big]}\\
&\gtrsim& \int_{[\frac{\pi}{8},\frac{\pi}{4}]^{d-1}} d\theta_1\cdots d\theta_{d-1} \int_2^{2^{n}} \frac{dr}{r^{4(H_1+\dots+H_d)-3d+1}} \int_0^{\frac12 r \cos \theta_1 }\cdots \int_0^{\frac12 r\sin(\theta_1)\cdots\sin(\theta_{d-1})}\frac{dr_1\cdots dr_d}{\{1+r_1^2+\cdots+r_d^2\}^{2\al}} \\
& &\gga^{H_0,n}_{t}(r) \gga^{H_0,n}_{t} \Bigg( \sqrt{r^2 \cos^2\theta_1 \Big(1-\frac{r_1}{r \cos \theta_1}\Big)^2+\cdots+r^2 \sin^2\theta_1\cdots\sin^2\theta_{d-1}\Big(1-\frac{r_d}{r \sin \theta_1\cdots\sin\theta_{d-1}}\Big)^2}\Bigg).
\end{eqnarray*}
\noindent
A quick view on the integration domain reveals that
$$r \geq \tilde{r}_{\theta}:= \sqrt{r^2 \cos^2\theta_1 \Big(1-\frac{r_1}{r \cos \theta_1}\Big)^2+\cdots+r^2 \sin^2\theta_1\cdots\sin^2\theta_{d-1}\Big(1-\frac{r_d}{r \sin \theta_1\cdots\sin\theta_{d-1}}\Big)^2} \geq \frac12 r \geq 1 \ .$$
Resorting to Lemma \ref{tech-lem-explos}, the (forthcoming) lower bound (\ref{lower-bound-gga-n}) leads to
\begin{eqnarray*}
\lefteqn{\mathbb{E}\Big[ \|\chi \cdot \<Psi2>^n(t,.)\|_{H^{-2\al}}^2\Big] \ \gtrsim \ \int_{[\frac{\pi}{8},\frac{\pi}{4}]^{d-1}} d\theta_1\cdots d\theta_{d-1}  \int_2^{2^{n}} \frac{dr}{r^{4(H_1+\dots+H_d)-3d+1}} }\nonumber\\
& &\hspace{2cm}  \int_0^{\frac12 r \cos \theta_1 }\cdots \int_0^{\frac12 r\sin(\theta_1)\cdots\sin(\theta_{d-1})}\frac{dr_1\cdots dr_d}{\{1+r_1^2+\cdots+r_d^2\}^{2\al}} \frac{1}{r^{8H_0}}(1-2e^{-t}+e^{-2t})^2\\
&\gtrsim& (1-2e^{-t}+e^{-2t})^2\bigg(\int_0^{\cos \frac{\pi}{8} }\cdots\int_0^{\sin (\frac{\pi}{8})^{d-1}}\frac{dr_1\cdots dr_d}{\{1+r_1^2+\cdots+r_d^2\}^{2\al}}\bigg)\bigg(\int_2^{2^{n}} \frac{dr}{r^{4(2H_0+H_1+\dots+H_d)-3d+1}}\bigg) \ .\label{low-bou-pr}
\end{eqnarray*}

\smallskip

\noindent
As $2H_0+H_1+\cdots+H_d \leq \frac{3}{4}d$, $4(2H_0+H_1+\dots+H_d)-3d+1\leq 1$,  and consequently
$$\int_2^{2^{n}} \frac{dr}{r^{4(2H_0+H_1+\dots+H_d)-3d+1}} \ \underset{n\to +\infty}{\longrightarrow} \ +\infty \ .$$
We then get the desired conclusion
$$\mathbb{E}\Big[ \|\chi \cdot \<Psi2>^n(t,.)\|_{H^{-2\al}}^2\Big]\underset{n\to +\infty}{\longrightarrow} \ +\infty.$$

\

\begin{lemma}\label{tech-lem-explos}
For all ${H_0}\in (0,1)$, $n\geq 1$, $t>0$ and $r \in [1,2^n]$,
\begin{equation}\label{lower-bound-gga-n}
\gga^{H_0,n}_t(r) \gtrsim \frac{1}{r^{4H_0}}(1-2e^{- t}+e^{-2 t})>0.
\end{equation} 
\end{lemma}

\begin{proof}
Remember that
$$|\gamma_t(\xi,r)|^2=\frac{1-2\cos(\xi t)e^{-r^2 t}+e^{-2r^2 t}}{r^4+\xi^2}\, .$$ 
It is clear that $\displaystyle\gga^{H_0,n}_{t}(r)=\int_{-4^n}^{4^n}  \, \frac{|\ga_t(\xi,r)|^2 }{|\xi|^{2H_0-1}}d\xi \geq\int_{-r^2}^{r^2}  \, \frac{|\ga_t(\xi,r)|^2 }{|\xi|^{2H_0-1}}d\xi \ .$
Now, by parity,
\begin{eqnarray*}
\int_{-r^2}^{r^2}  \, \frac{|\ga_t(\xi,r)|^2 }{|\xi|^{2H_0-1}}d\xi&=&2\int_{0}^{r^2}  \, \frac{|\ga_t(\xi,r)|^2 }{|\xi|^{2H_0-1}}d\xi \\
&=&\frac{2}{r^{4H_0}}\int_0^1 \frac{1-2\cos(r^2\xi t)e^{-r^2 t}+e^{-2r^2 t}}{(1+\xi^2)\xi^{2H_0-1}}d\xi\\
&\geq& \frac{2}{r^{4H_0}}\left(\int_0^1 \frac{d\xi}{(1+\xi^2)\xi^{2H_0-1}}\right)(1-2e^{-r^2t}+e^{-2r^2t})>0.
\end{eqnarray*}
Indeed, if we set for all $t\geq 0, r\geq 1,$ $$h(t,r):=1-2e^{-r^2t}+e^{-2r^2t},$$
it is quite easy to check that, when $r\geq1$, $h(.,r)$
 is strictly increasing on $[0,+\infty)$ and, since $h(0,r)=0$, for every $t>0$, $h(t,r)=1-2e^{-r^2t}+e^{-2r^2t}>0.$
 Moreover, we can verify that if $t\geq 0,$ $h(t,.)$ is increasing on $[1,+\infty)$ that provides the conclusion.

\end{proof}

\section{Proof of the main results}\label{end}
\subsection{Proof of Theorem~\ref{resu}}
Suppose that $ d\geq 1$. Let $p\geq2$ and $\be$ be such that $0< \beta < 2H_0+\sum_{i=1}^{d}H_i-d$ and $\frac{d}{2p}<1+\frac{\be}{2}$.\\
Recall that for every $T\geq0,$
\begin{equation}
X^{\be,p}(T):=\mathcal{C}([0,T]; \mathcal{W}^{\be,p}(\mathbb{R}^d)) \, .
\end{equation}
$i)$ The statement of Proposition~\ref{sto} with $\al:=-\be$ and $\chi:=\rho$ ensures the existence of a measurable set $\tilde{\Omega}$ of measure one on which $\rho \<Psi>(\omega) \in X^{\be,p}(T)$. Now, the well-posedness result of Theorem~\ref{resu} comes from the application of Theorem~\ref{thm:regular} (in an almost sure way) to $\luxor:=\rho\<Psi>$.

\smallskip
\noindent
$ii)$ By reproducing the arguments that can be found in the proof of~\cite[Theorem 1.7]{deya-wave}, we see that the convergence property is a consequence from the continuity of $\Gamma_{T,\luxor}$ with respect to $\luxor$ (along~\eqref{bound2}) and the almost sure convergence of $\chi \<Psi>_n$ to $\chi \<Psi>$.

\subsection{Proof of Theorem~\ref{resu1}}

Suppose that $d\geq 1$ and $p\geq2$ verifies that $\frac{d}{2p}\leq \frac{3}{4}.$
Assume that $2H_0+\sum_{i=1}^{d}H_i\leq d$.
Fix $\al>0$ such that $$d-\bigg(2H_0+\sum_{i=1}^{d}H_i\bigg)<\al<\frac{1}{4}.$$
As $\frac{d}{2p}\leq \frac{3}{4}$, observe that $\alpha<\frac{1}{4}\leq 1-\frac{d}{2p}$ and we can pick $\al<\be<\min\bigg(2-\al-\frac{d}{p}, 2-2\al\bigg).$\\
Recall that
\begin{equation}
\mathcal{R}_{\al,p}:=L^\infty\big([0,T];\mathcal{W}^{-\alpha,p}(\mathbb{R}^d)\big)\times L^\infty\big([0,T];\mathcal{W}^{-2\alpha,p}(\mathbb{R}^d)\big) \, .
\end{equation}
$i)$ Proposition~\ref{sto} and Proposition~\ref{sto1} applied with $\al>0$ guarantees the existence of a measurable set $\tilde{\Omega}$ of measure one on which $(\rho \<Psi>(\omega),\rho^2\<Psi2>(\omega)) \in \mathcal{R}_{\al,p}$. Now, the statement of Theorem~\ref{resu1} results from the application of Theorem~\ref{thm:regular} (in an almost sure way) to $(\luxor,\cherry):=(\rho\<Psi>,\rho^2\<Psi2>)$.

\smallskip
\noindent
$ii)$ Again, we remark that the convergence property is a consequence from the continuity of $\Gamma_{T,\luxor,\cherry}$ with respect to $(\luxor,\cherry)$ and the almost sure convergence of $(\chi \<Psi>_n,\chi^2 \<Psi2>_n)$ to $(\chi \<Psi>,\chi^2\<Psi2>)$.

\subsection{Proof of Theorem~\ref{resu2}}

Suppose that $d=2$ and $p\geq2.$
Let $(H_0,H_1,H_2)\in (0,1)^{3}$ be such that 
\begin{equation}
0<H_1<\frac34 \quad , \quad 0<H_2< \frac34 \quad , \quad \frac{3}{2} < 2H_0+H_1+H_2 \leq  \frac74 \ .
\end{equation}
Fix $\al>0$ such that $$2-(2H_0+H_1+H_2)<\al<\frac{1}{2}.$$
As $p\geq 2$, observe that $\alpha<\frac{1}{2}\leq 1-\frac{2}{2p}$ and we can pick $\al<\be<\min\bigg(2-\al-\frac{d}{p}, 2-2\al\bigg).$\\
$i)$ Proposition~\ref{sto} and Proposition~\ref{sto1} applied with $\al>0$ guarantees the existence of a measurable set $\tilde{\Omega}$ of measure one on which $(\rho \<Psi>(\omega),\rho^2\<Psi2>(\omega)) \in \mathcal{R}_{\al}$. Again, Theorem~\ref{resu1} results from the application of Theorem~\ref{thm:regular} (in an almost sure way) to $(\luxor,\cherry):=(\rho\<Psi>,\rho^2\<Psi2>)$.

\smallskip
\noindent
$ii)$ As before, the convergence property is a consequence from the continuity of $\Gamma_{T,\luxor,\cherry}$ with respect to $(\luxor,\cherry)$ and the almost sure convergence of $(\chi \<Psi>_n,\chi^2 \<Psi2>_n)$ to $(\chi \<Psi>,\chi^2\<Psi2>)$.

\section{Appendix}
\subsection{Technical lemmas}
\smallskip

In this subsection, we state the three technical lemmas at the core of the proof of Proposition \ref{prop:stoch-constr}.

\begin{lemma}\label{lem:bou-l-h}
For all $H=(H_0,H_1,H_2)\in (0,1)^3$, $\varepsilon\in (0,H_0)$, $0\leq n\leq m$, $0\leq s,t,u \leq 1$ and $\eta\in \R^2$, it holds that
\begin{equation}\label{bou-l-1}
| L^{H,(m,m)}_{t,t}(\eta) \big| \lesssim K^{H_{\varepsilon,0}}(\eta) 
\end{equation}
and 
\begin{equation}\label{bou-l-2}
| L^{H,((n,m),m)}_{(s,t),u}(\eta) \big|\lesssim 2^{-n\varepsilon}|t-s|^{\frac{\varepsilon}{2}} \big\{ K^{H_{\varepsilon,0}}(\eta)+K^{H_{\varepsilon,0,1}}(\eta)+K^{H_{\varepsilon,0,2}}(\eta) \big\} \ ,
\end{equation}
where  $H_{\varepsilon,0} := (H_0-\varepsilon,H_1,H_2)$, $H_{\varepsilon,0,1} := (H_0-\varepsilon,H_1-\varepsilon,H_2)$, $H_{\varepsilon,0,2} := (H_0-\varepsilon,H_1,H_2-\varepsilon)$, and the proportional constants do no depend on $(n,m)$, $(s,t)$, $u$ and $\eta$.
\end{lemma}

\begin{proof}
The first inequality is a straight consequence of Corollary \ref{tec} :
\begin{eqnarray*}
\big| L^{H,(m,m)}_{t,t}(\eta) \big|&=&\Bigg| \frac{1}{|\eta_1|^{2H_1-1} |\eta_2|^{2H_2-1}} \int_{(\xi,\eta)\in D^{m} } d\xi \, \frac{|\ga_{t}(\xi,|\eta|)|^2}{|\xi|^{2H_0-1}}\Bigg|  \\
&\lesssim& \frac{1}{|\eta_1|^{2H_1-1} |\eta_2|^{2H_2-1}} \int_{\mathbb{R} } d\xi \, \frac{|\ga_{t}(\xi,|\eta|)|^2}{|\xi|^{2H_0-1}}\lesssim K^{H_{\varepsilon,0}}(\eta).
\end{eqnarray*}
The second one is a bit more technical. Recall that one has
$$ L^{H,((n,m),m)}_{(s,t),u}(\eta) := \frac{1}{|\eta_1|^{2H_1-1} |\eta_2|^{2H_2-1}} \int_{(\xi,\eta)\in D^{n,m} \cap D^{m}} d\xi \, \frac{\ga_{s,t}(\xi,|\eta|) \overline{\ga_{u}(\xi,|\eta|)}}{|\xi|^{2H_0-1}}  $$
which leads to
\begin{eqnarray*}
\big| L^{H,((n,m),m)}_{(s,t),u}(\eta)\big| &\lesssim &  \frac{1}{|\eta_1|^{2H_1-1} |\eta_2|^{2H_2-1}} \int_{(\xi,\eta)\in D^{n,m}} d\xi \, \frac{|\ga_{s,t}(\xi,|\eta|)|| \ga_{u}(\xi,|\eta|)|}{|\xi|^{2H_0-1}}\\
&\lesssim& \frac{1}{|\eta_1|^{2H_1-1} |\eta_2|^{2H_2-1}}\int_{ |\xi| \geq 2^{2n}} d\xi \, \frac{|\ga_{s,t}(\xi,|\eta|)|| \ga_{u}(\xi,|\eta|)|}{|\xi|^{2H_0-1}}\\
&&+ \frac{1}{|\eta_1|^{2H_1-1} |\eta_2|^{2H_2-1}}\mathbbm{1}_{ |\eta| \geq 2^{n}}\int_{\mathbb{R}} d\xi \, \frac{|\ga_{s,t}(\xi,|\eta|)|| \ga_{u}(\xi,|\eta|)|}{|\xi|^{2H_0-1}}.
\end{eqnarray*}
Let us call $\mathbb{I}_{m,n}(s,t,u)$ (resp. $\mathbb{II}_{m,n}(s,t,u)$) the first (resp. the second) integral.
Since $0<~\varepsilon<~H_0$, Corollary \ref{tec} combined with Cauchy-Schwarz inequality entails:
$$\mathbb{I}_{m,n}(s,t,u) \lesssim 2^{-2n\varepsilon}|t-s|^{\frac{\varepsilon}{2}}K^{H_{\varepsilon,0}}(\eta),$$
whereas, keeping in mind the identity $$\{ |\eta|\geq 2^n\} \subset \{  |\eta|\geq 2^n, |\eta_1|\geq \frac{1}{\sqrt{2}}|\eta| \}\cup\{  |\eta|\geq 2^n, |\eta_2|\geq \frac{1}{\sqrt{2}}|\eta| \},$$
it holds
\begin{eqnarray*}
\frac{1}{|\eta_1|^{2H_1-1} |\eta_2|^{2H_2-1}}\mathbbm{1}_{ |\eta| \geq 2^{n}}\lesssim 2^{-2n\varepsilon}\frac{1}{|\eta_1|^{2(H_1-\varepsilon)-1} |\eta_2|^{2H_2-1}}+ 2^{-2n\varepsilon}\frac{1}{|\eta_1|^{2H_1-1} |\eta_2|^{2(H_2-\varepsilon)-1}}
\end{eqnarray*}
which, thanks to Corollary \ref{tec}, immediately implies
$$\mathbb{II}_{m,n}(s,t,u)\lesssim 2^{-n\varepsilon}|t-s|^{\varepsilon} \big\{K^{H_{\varepsilon,0,1}}(\eta)+K^{H_{\varepsilon,0,2}}(\eta) \big\}.$$
\end{proof}

The lemma below brings back computations on compact domains and its proof can be found in \cite[Lemma 2.6]{deya3}.
\begin{lemma}\label{lem:handle-chi-correc}
Let $\chi:\R^d \to \R$ be a test function and fix $\si\in \R$. Then, for every $p\geq 1$ and for all $\eta_1,\ldots,\eta_p\in \R^d$, it holds that
$$\bigg| \int_{\R^d} dx \, \prod_{i=1}^p \int_{(\R^d)^2} \frac{d\la_i d\lati_i}{\{1+|\la_i|^2\}^{\frac{\si}{2}}\{1+|\lati_i|^2\}^{\frac{\si}{2}}} e^{\imath \langle x,\la_i-\lati_i\rangle} \widehat{\chi}(\la_i-\eta_i)\overline{\widehat{\chi}(\lati_i-\eta_i)}  \bigg| \lesssim \prod_{i=1}^p \frac{1}{\{1+|\eta_i|^2\}^\si} \, ,$$
where the proportional constant only depends on $\chi$ and $\si$. 
\end{lemma}

We end this subsection by a highly technical lemma that permits us to construct the second order stochastic process $\<Psi2>$ when $d=2$ in the roughest case.
\begin{lemma}\label{lem:tech-ordre-deux}
For all $H=(H_0,H_1,H_2), \Hti=(\Hti_0,\Hti_1,\Hti_2) \in (0,1)^3$ verifying
\begin{equation}\label{constraint-h-i-lem}
0<H_1, \Hti_1<\frac34 \quad , \quad 0<H_2,\Hti_2< \frac34 \quad , \quad \frac{3}{2} < 2H_0+H_1+H_2 \leq  \frac74 \quad ,\quad \frac{3}{2} < 2\Hti_0+\Hti_1+\Hti_2 \leq  \frac74  \ ,
\end{equation} 
and any
\begin{equation}\label{tech-cond-al}
\al\in \big(\max\big(2-(2H_0+H_1+H_2),2-(2\Hti_0+\Hti_1+\Hti_2)\big),\frac12\big) \ ,
\end{equation}
it holds that
$$
\iint_{(\R^2)^2} \frac{d\eta d\etati}{\{1+|\eta-\etati|^2\}^{2\al}} K^H(\eta) K^{\Hti}(\etati) \ < \ \infty\ .
$$
\end{lemma}

\begin{proof}

\

\noindent
It is a simple adaptation of the proof of \cite[Lemma 3.3]{deya-wave-2}.
\end{proof}

\subsection{Estimation of the constant $c_1$ from Proposition \ref{equi}}
Let us compute the value of $c_1$. By developing the integrand in Taylor series, we write
\begin{align}
\int_0^1 \frac{x^{\al+\kappa-1}}{1+x^2}dx+\int_1^{+\infty} \frac{x^{\al-1}}{1+x^2}dx&=\int_0^1 \frac{x^{\al+\kappa-1}}{1+x^2}dx+\int_0^{1} \frac{x^{1-\al}}{1+x^2}dx\nonumber\\
&=\sum_{n=0}^{+\infty}(-1)^n \int_0^1 x^{2n+\al+\ka-1}dx +\sum_{n=0}^{+\infty}(-1)^n \int_0^1 x^{2n+1-\al}dx \nonumber\\
&=\sum_{n=0}^{+\infty}\frac{(-1)^n}{2n+\al+\ka}+\sum_{n=0}^{+\infty}\frac{(-1)^n}{2n+2-\al}\nonumber\\
&=\sum_{n=0}^{+\infty}\frac{2}{(4n+\al+\ka)(4n+2+\al+\ka)}+\sum_{n=0}^{+\infty}\frac{2}{(4n+2-\al)(4n+4-\al)}\nonumber\\
&=\frac{1}{8}\sum_{n=0}^{+\infty}\frac{1}{(n+\frac{\al+\ka}{4})(n+\frac{\al+\ka+2}{4})}+\frac{1}{8}\sum_{n=0}^{+\infty}\frac{1}{(n+\frac{2-\al}{4})(n+1-\frac{\al}{4})}. \label{laste}
\end{align}

\noindent
To continue the computation, we need to introduce two well-known special functions, namely:
\begin{definition}
The Gamma function $\Gamma$ and the digamma function $\Psi$ are defined for all $x>0$ by the formulas
$$\Gamma(x)=\int_0^{+\infty}e^{-t}t^{x-1}dt, \quad \Psi(x)=\frac{\Gamma'(x)}{\Gamma(x)}.$$
\end{definition}

\noindent
We are now in a position to recall the following classical result from analytic number theory:

\begin{lemma}\label{gamma}
For all $a,b>0$, it holds that:
$$\sum_{n=0}^{+\infty}\frac{1}{(n+a)(n+b)}=\frac{\Psi(b)-\Psi(a)}{b-a}.$$
\end{lemma}

\noindent
Coming back to (\ref{laste}), with the help of Lemma \ref{gamma}, we deduce
$$\int_0^1 \frac{x^{\al+\kappa-1}}{1+x^2}dx+\int_1^{+\infty} \frac{x^{\al-1}}{1+x^2}dx=\frac{1}{4}\Big[\Psi\left(\frac{\al+\ka+2}{4}\right)-\Psi\left(\frac{\al+\ka}{4}\right)+\Psi\left(\frac{4-\al}{4}\right)-\Psi\left(\frac{2-\al}{4}\right)\Big]$$
and, finally, for your viewing pleasure,
$$c_1=\frac{1}{4\ka}\Bigg[\frac{\Gamma'(\frac{\al+\ka+2}{4})}{\Gamma(\frac{\al+\ka+2}{4})}+\frac{\Gamma'(\frac{4-\al}{4})}{\Gamma(\frac{4-\al}{4})}-\frac{\Gamma'(\frac{\al+\ka}{4})}{\Gamma(\frac{\al+\ka}{4})}-\frac{\Gamma'(\frac{2-\al}{4})}{\Gamma(\frac{2-\al}{4})}\Bigg].$$

\end{document}